\documentclass[11pt, letterpaper]{article}
\usepackage[centertags]{amsmath}
\usepackage{graphicx}
\usepackage{amssymb}
\usepackage{indentfirst}
\usepackage{amsmath}
\usepackage{amsthm}
\usepackage{color}
\usepackage{tikz-cd}
\usepackage{relsize}
\usepackage{transparent}
\usepackage{mathtools}

  \newenvironment{proof of the claim}{\emph{ \quad  Proof of the Claim.}}{\hspace{\stretch{1}}\rule{1ex}{1ex}}

           \theoremstyle{plain}
            \newtheorem{theorem}{Theorem}[section]
            
            \newtheorem{lemma}[theorem]{Lemma}
            \newtheorem{corollary}[theorem]{Corollary}   
            \newtheorem{question}{Question}
            
            \newtheorem{proposition}[theorem]{Proposition}                                      
                        
            \theoremstyle{definition}
            \newtheorem{assumption}[theorem]{Assumption}
            \newtheorem{procedure}[theorem]{Procedure}
            \newtheorem{construction}[theorem]{Construction}
            \newtheorem{definition}[theorem]{Definition}
            \newtheorem{remark}[theorem]{Remark}
            \newtheorem{example}[theorem]{Example}

\renewcommand{\ll }{\langle\hspace{-.7mm}\langle }
\newcommand{\rr }{\rangle\hspace{-.7mm}\rangle }

\title{Cohomology of group theoretic Dehn fillings \textrm{I}: Cohen-Lyndon type theorems}
\author{Bin Sun}
\date{}
\begin{document}
\newcommand{\Addresses}{{
  \bigskip
  \footnotesize

  Bin Sun, \textsc{Mathematic Institute, University of Oxford, Oxford, UK, OX2 6GG}\par\nopagebreak
  \textit{E-mail address}, \texttt{bin.sun@maths.ox.ac.uk}
  }}
\maketitle


\begin{abstract}
This is the first paper of two papers in a row aiming to study cohomology of group theoretic Dehn fillings. In the present paper, we prove a particular free product structure, which is termed the Cohen-Lyndon property, of Dehn filling kernels. As an application, we describe the structure of relative relation modules of Dehn fillings.

\end{abstract}


\section{Introduction}

\paragraph{1.1 Dehn surgery of $3$-manifolds.} In $3$-dimensional topology, Dehn surgery is an operation of modifying a $3$-manifold by cutting off a solid torus and then gluing it back in a different way. The Lickorish-Wallace theorem, which states that every closed connected orientable $3$-manifold can be obtained from the $3$-dimensional sphere by performing finitely many Dehn surgeries, serves as a motivation of the study of Dehn surgeries.

The second step of the surgery, called \textit{Dehn filling}, can be formalized as follows. Let $M$ be a $3$-manifold with toral boundary. Topologically distinct ways of gluing a solid torus to $M$ are parametrized by free homotopy classes of essential simple closed curves on $\partial M$, called \textit{slopes}. For a slope $s$, the Dehn filling $M(s)$ is obtained by attaching a solid torus $S^1\times D^2$ to $\partial M$ such that $\partial D^2$ is mapped to a curve of the slope $s$. The following is a special case of Thurston's hyperbolic Dehn filling theorem.
\begin{theorem}[{\cite[Theorem {[TH1]}]{thurston1983three}}]\label{thm. thurston}
Let $M$ be a compact orientable $3$-manifold with toral boundary such that $M\setminus\partial M$ admits a complete finite-volume hyperbolic structure. Then $M(s)$ is hyperbolic for all but finitely many slopes $s$.
\end{theorem}

\paragraph{1.2 Group theoretic Dehn fillings.} In group theoretic settings, Dehn filling can be generalized as follows. Let $G$ be a group, let $H$ be a subgroup of $G$, and let $N$ be a normal subgroup of $H$. The \textit{group theoretic Dehn filling} associated with the data $(G,H,N)$ is the process of forming the quotient group $G/\ll N \rr$, where $\ll N \rr$ is the normal closure of $N$ in $G$. 

Under the assumptions of Theorem \ref{thm. thurston}, let $G=\pi_1(M)$. The natural map $\pi_1(\partial M)\rightarrow \pi_1(M)$ is injective. We think of $\pi_1(\partial M)$ as a subgroup of $\pi_1(M)$ and let $H=\pi_1(\partial M)$. Let $N\lhd H$ be the subgroup generated by the slope $s$. Then $G/\ll N \rr=\pi_1(M(s))$ by the Seifert-van Kampen theorem. 

Dehn filling is a fundamental tool in group theory. The solution of the virtually Haken conjecture uses Dehn fillings of hyperbolic groups \cite{agol2013virtual}. For a large number of relatively hyperbolic groups, Dehn fillings are used to prove the Farrell-Jones conjecture \cite{antolin2017farrell} and solve the isomorphism problem \cite{dahmani2018recognizing}. By considering Dehn fillings of hyperbolically embedded subgroups, \cite{dahmani2017hyperbolically} constructs purely pseudo-Anosov normal subgroups of mapping class groups. Other applications of Dehn fillings can be found in \cite{agol2016alternate,groves2016boundaries}.

In group theoretic settings, Thurston's theorem was first generalized by Osin \cite{osin2007peripheral}, and independently by Groves-Manning \cite{groves2008dehn} to Dehn fillings of peripheral subgroups of relatively hyperbolic groups. More recently, Dahmani-Guirardel-Osin \cite{dahmani2017hyperbolically} proved an analog of Thurston's theorem in the more general settings of groups with hyperbolically embedded subgroups (see Theorem \ref{thm. first part of simple DGO} below and the discussion afterwards). We discuss here some examples and refer to Section \ref{subsec.wh} for the definition. We use $H\hookrightarrow_h G$ to indicate that $H$ is a hyperbolically embedded subgroup of $G$.

\begin{example}\label{eg. 1}
If $H$ is a peripheral subgroup of a relatively hyperbolic group $G$, then $H\hookrightarrow_h G$ \cite[Proposition 2.4]{dahmani2017hyperbolically}. For example,
\begin{enumerate}
\item[(a)] if a group $G$ decomposes as a free product $G=A\ast B$, then we have $A\hookrightarrow_h G$ and $B\hookrightarrow_h G$ \cite{osin2006relatively};
\item[(b)] under the assumptions of Theorem \ref{thm. thurston}, we have $\pi_1(\partial M)\hookrightarrow_h\pi_1(M)$ \cite{bowditch2012relatively,farb1998relatively}.
\end{enumerate}
\end{example}

\begin{example}\label{eg. 2} 
Let $G$ be a group acting acylindrically on a Gromov hyperbolic space and let $g$ be a loxodromic element of $G$. Then there exists a maximal virtually cyclic subgroup $E(g)\leqslant G$ containing $g$ such that $E(g)\hookrightarrow_h G$ \cite[Corollary 2.9]{dahmani2017hyperbolically}. In particular, if $G$ is a hyperbolic group (resp. the mapping class group of a punctured closed orientable surface \cite[Theorem 2.19]{dahmani2017hyperbolically}, outer automorphism group of a finite rank non-abelian free group \cite[Theorem 2.20]{dahmani2017hyperbolically}) and $g$ is a loxodromic (resp. pseudo-Anosov, fully irreducible) element, then $E(g)\hookrightarrow_h G$.
\end{example}
Other examples of hyperbolically embedded subgroups can be found in \cite{dahmani2017hyperbolically}.

\begin{theorem}[{\cite[Theorem 2.27]{dahmani2017hyperbolically}}]\label{thm. first part of simple DGO}
Let $G$ be a group with a subgroup $H\hookrightarrow_h G$. Then there exists a finite set $\mathcal{F}\subset H\setminus\{1\}$ such that if $N\lhd H$ and $N\cap \mathcal{F}=\emptyset$, then the natural homomorphism $H/N\rightarrow G/\ll N \rr $ maps $H/N$ injectively onto a hyperbolically embedded subgroup of $G/\ll N \rr$.
\end{theorem}

In fact, Theorem \ref{thm. first part of simple DGO} generalizes Theorem \ref{thm. thurston}: let $M$ be as in Theorem \ref{thm. thurston}. By Example \ref{eg. 1}, for all but finitely many slopes $s$ on $\partial M$, Theorem \ref{thm. first part of simple DGO} implies that $\pi_1(M(s))$ is word-hyperbolic and \cite[Corollary 1.11]{groves2018dehn} implies that $\pi_1(M(s))$ is one-ended. The Geometrization Conjecture, proved by Perelman, then implies that $M(s)$ admits a hyperbolic structure.

\paragraph{1.3 Motivation: a question on group cohomology.}
Note that in the settings of Thurston's theorem, i.e., if $G=\pi_1(M),H=\pi_1(\partial M),$ and $M(s)$ admits a hyperbolic structure, we have 
$$H^{\ast}(G/\ll N \rr;\cdot)\cong H^{\ast}(\pi_1(M(s));\cdot),$$
which can be computed via $M(s)$. Indeed, as $M(s)$ admits a hyperbolic structure, the universal cover of $M(s)$ is $\mathbb{H}^3$, which is contractible, and thus $M(s)$ is a model of $K(G/\ll N \rr,1)$. 

However, there are no analogous methods for Dehn fillings of hyperbolically embedded subgroups. The main question motivating our research is the following.
\begin{question}\label{question. main question}
For a group $G$ with a subgroup $H\hookrightarrow_h G$ and a normal subgroup $N\lhd H$, what can be said about $H^{\ast}(G/\ll N \rr;\cdot)$?
\end{question}


In this series of two papers, we answer this question and discuss applications. The first task is to understand the structure of $\ll N \rr$, which is solved by the present paper. In the second paper \cite{sun2019cohomologyii}, we will combine structural results obtained in this paper and the Lyndon-Hochschild-Serre spectral sequence to compute $H^{\ast}(G/\ll N \rr; \cdot)$, and then we will study cohomological properties of $G/\ll N \rr$ and discuss some applications to simplicial volume and acylindrically hyperbolic groups.

\section{Main results}
\paragraph{2.1 Cohen-Lyndon type theorems for $\ll N \rr$.}
In general, $\ll N \rr$ does not need to have any particular structure. Nevertheless, it turns out that if $N$ avoids a finite set of bad elements, then $\ll N \rr$ enjoys a nice free product structure. In order to state our main results, we introduce the following terminology.

\begin{definition}
Let $G$ be a group with a subgroup $H\hookrightarrow_h G$. We say that a property $P$ holds \textit{for all sufficiently deep} normal subgroups $N\lhd H$ if there exists a finite set $\mathcal{F}\subset H\setminus\{1\}$ such that $P$ holds for all normal subgroups $N\lhd H$ with $N\cap\mathcal{F}=\emptyset$.
\end{definition}

\begin{definition}\label{simple CLp}
Let $G$ be a group with a subgroup $H$ and let $N\lhd H$. We say that the triple $(G,H,N)$ has the \textit{Cohen-Lyndon property} if there exists a left transversal $T$ of $H\ll N \rr$ in $G$ such that $\ll N \rr$ is the free product of its subgroups $N^t=tNt^{-1}$ for $t\in T$, denoted as 
$$\ll N \rr=\prod^{\ast}_{t\in T}N^t.$$
\end{definition}

The latter definition is motivated by the following result \cite[Theorem 4.1]{cohen1963free}, which was later generalized by \cite[Theorem 1.1]{edjvet1987cohen} to free products of locally indicable groups.

\begin{theorem}[Cohen-Lyndon]\label{classical CL}
Let $F$ be a free group and let $C$ be a maximal cyclic subgroup of $F$. Then for all $f\in C\setminus\{1\}$, the triple $(F,C,\langle f \rangle)$ has the Cohen-Lyndon property.
\end{theorem}

Note that $F$ is hyperbolic and $f$ has infinite order. By Example \ref{eg. 2}, we have $C=E(f)\hookrightarrow_h F$ and thus the above theorem fits in the general framework of group theoretic Dehn fillings. For general hyperbolically embedded subgroups, a weak version of the Cohen-Lyndon property is given in \cite[Theorem 2.27]{dahmani2017hyperbolically}.

\begin{theorem}[{Dahmani-Guirardel-Osin}]\label{simple DGO}
Let $G$ be a group with a subgroup $H\hookrightarrow_h G$. Then for all sufficiently deep $N\lhd H$, we have
$$\ll N \rr=\prod^{\ast}_{t\in T}N^t$$
for some subset $T\subset G$.
\end{theorem}

The main difference between Theorems \ref{simple DGO} and \ref{classical CL} is that in Theorem \ref{simple DGO}, $T$ is just some subset of $G$, instead of being a left transversal of $H \ll N \rr$ in $G$. Our result improves Theorem \ref{simple DGO}.
\begin{theorem}\label{simmain}
Suppose that $G$ is a group with a subgroup $H\hookrightarrow_h G$. Then $(G,H,N)$ has the Cohen-Lyndon property for all sufficiently deep $N\lhd H$.
\end{theorem}

In the special case where $G$ and $H$ are finitely generated and $G$ is hyperbolic relative to $H$, Theorem \ref{simmain} is proved in \cite[Theorem 4.8]{groves2016boundaries}. The proofs of \cite[Theorem 7.15]{dahmani2017hyperbolically} and \cite[Theorem 4.8]{groves2016boundaries} use windmills, very rotating families, and spiderwebs. It is possible to prove Theorem \ref{simmain} with these notions. However, the present paper provides a proof of a different flavor, using surgery on van Kampen diagrams and geodesic polygons of Cayley graphs.

\begin{remark}\label{rem. general version of CL}
In fact, we prove Theorem \ref{simmain} in more general settings of a group with a family of \textit{weakly hyperbolically embedded subgroups} (see Definition \ref{hyperbolically embedded}). As an application, we obtain Cohen-Lyndon type theorems for graphs of groups, e.g., amalgamated free products and HNN-extensions (see Corollaries \ref{graph of groups}, \ref{amalproduct}, and \ref{HNN}).
\end{remark}

Combining Theorem \ref{simmain} and Example \ref{eg. 2}, we obtain:

\begin{corollary}
Let $G$ be a group acting acylindrically on a Gromov hyperbolic space, and let $g\in G$ be a loxodromic element. Then $(G,E(g),N)$ has the Cohen-Lyndon property for all sufficiently deep $N\lhd E(g)$.
\end{corollary}

In the case where $G=F$ and $H=C$, we recover Theorem \ref{classical CL} for sufficiently deep (but not all) $\langle f \rangle\lhd C$. In the case where $G$ is a free product of locally indicable groups, by considering the action of $G$ on the corresponding Bass-Serre tree, we recover \cite[Theorem 1.1]{edjvet1987cohen} for sufficiently deep normal subgroups.

\paragraph{2.2 Structure of relative relation modules.} Let $Rel(G,\ll N \rr)$ and $Rel(H,N)$ be the relative relation modules of the exact sequences
$$1\rightarrow \ll N \rr \rightarrow G\rightarrow Q\rightarrow 1$$
and
$$1\rightarrow N\rightarrow H\rightarrow R\rightarrow 1,$$
respectively, i.e. $Rel(G,\ll N \rr)$ (resp. $Rel(H,N)$) is the $\mathbb{Z}Q$-module (resp. $\mathbb{Z}R$-module) whose base set is the abelianization of $\ll N \rr$ (resp. $N$) and the $Q$-action (resp. $R$-action) is induced by conjugation. If $G$ is free, then $Rel(G,\ll N \rr)$ is called a \textit{relation module}. For sufficiently deep $N$, it follows immediately from Theorem \ref{thm. first part of simple DGO} that the natural map identifies $R$ with a subgroup of $Q$. We can then further identify $\mathbb{Z}R$ with a subring of $\mathbb{Z}Q$. Thus, given any $\mathbb{Z}R$-module $A$, it makes sense to talk about the \textit{induced module} of $A$ from $\mathbb{Z}R$ to $\mathbb{Z}Q$, which is denoted by $Ind^Q_R A=\mathbb{Z}Q\bigotimes_{\mathbb{Z}R}A$.

If $G=F$ and $H=C$, Theorem \ref{classical CL} directly implies 
$$Rel(F,\ll f \rr)\cong\mathbb{Z}[F/C\ll f \rr]\cong Ind^Q_R \mathbb{Z}\cong Ind^Q_R Rel(C,\langle f \rangle)$$
as $\mathbb{Z}Q$-modules. In general, we have the following corollary to Theorem \ref{simmain}.

\begin{corollary}\label{module}
Let $G$ be a group with a subgroup $H\hookrightarrow_h G$. Then for all sufficiently deep $N\lhd H$, there is an isomorphism of $\mathbb{Z}Q$-modules
\begin{equation}\label{iso}
Rel(G,\ll N \rr)\cong Ind^Q_RRel(H,N).
\end{equation}
\end{corollary}

\begin{remark}
Merely knowing that $\ll N \rr = \prod^{\ast}_{t\in T}N^t$ for some subset $T\subset G$ is not enough to guarantee \eqref{iso}. For example, let $G$ be any abelian group and let $H$ be a proper subgroup of $G$. Then for any subgroup $N$ of $H$, $\ll N \rr = N = \prod^{\ast}_{t\in\{1\}}N^t$. But $Rel(G,\ll N \rr)$ (resp. $Rel(H,N)$) is a $\mathbb{Z}Q$-module (resp. $\mathbb{Z}R$-module) with the trivial $Q$-action (resp. $R$-action) and thus $Rel(G,\ll N \rr)\not\cong Ind^Q_RRel(H,N)$.
\end{remark}

\paragraph{2.3 Organization of the paper.}
The strategy of this paper is to deal with the more general setting where $G$ has a family of weakly hyperbolically embedded subgroups and prove the Cohen-Lyndon property in this case. Theorem \ref{simmain} and Corollary \ref{module} are simple consequences of the general results. Necessary preliminaries are provided in Section \ref{prelim}, which is divided into four subsections. Section \ref{subsec.vk} recalls the definition and basic properties of van Kampen diagrams. Section \ref{subsec.wh} surveys the notion of (weakly) hyperbolically embedded subgroups. The proof of the Cohen-Lyndon property relies heavily on surgeries on van Kampen diagrams and geodesic polygons of Cayley graphs. Section \ref{subsec.ic} is devoted to a concept called isolated components, which is vital to surgeries, and Section \ref{sec.surgery} collects some results about surgeries on van Kampen diagrams. In Section \ref{sec.transversal}, we construct particular transversals (with the aid of Zorn's lemma). Section \ref{proof} states and proves the main theorem of this paper in the general case by using the transversals constructed in Section \ref{sec.transversal}. Finally, Section \ref{sec.rm} analyzes the structure of the relative relation modules of Dehn fillings.

\paragraph{Acknowledgement.} This work was done when I was a graduate student at Vanderbilt University. I would like to thank my supervisor, Professor Denis Osin, for the priceless discussions. This paper would not have been written without his help. I would also like to thank the anonymous referee for the useful comments on an early version of this paper and for showing me an alternative proof of Theorem \ref{simmain} using very rotating families.

\section{Preliminaries}\label{prelim}
This section contains a brief discussion of concepts and tools used to prove the main theorem. Our main reference is \cite{dahmani2017hyperbolically}.

\subsection{Van Kampen diagrams}\label{subsec.vk}
Let $G$ be a group given by the presentation
\begin{equation}\label{pre0}
G=\langle\mathcal{A}\mid\mathcal{R}\rangle,
\end{equation}
where $\mathcal{A}$ is a symmetric set of letters and $\mathcal{R}$ is a symmetric set of words in $\mathcal{A}$ (i.e., for every $w\in\mathcal{R}$, every cyclic shift of $w$ or $w^{-1}$ belongs to $\mathcal{R}$).

A \textit{van Kampen diagram} $\Delta$ over \eqref{pre0} is a finite oriented connected planar $2$-complex with labels on its oriented edges such that
\begin{enumerate}
\item[(a)] each oriented edge of $\Delta$ is labeled by a letter in $\mathcal{A}\cup\{1\}$;
\item[(b)] if an oriented edge $e$ of $\Delta$ has label $a\in\mathcal{A}\cup\{1\}$, then $e^{-1}$ has label $a^{-1}$, where $e^{-1}$ (resp. $a^{-1}$) is the inverse of $e$ (resp. $a$).
\end{enumerate}

Here, $1$ is identified with the empty word over $\mathcal{A}$ and thus $1=1^{-1}$. By convention, the empty word of $\mathcal{A}$ represents the identity of $G$.

Let $p=e_1\cdot\cdot\cdot e_k$ be a path in a van Kampen diagram over \eqref{pre0} or in the Cayley graph $\Gamma(G, \mathcal{A})$. The initial vertex (resp. terminal vertex) of $p$ is denoted as $p^-$ (resp. $p^+$). The \textit{label} of $p$, denoted as $Lab(p)$, is obtained by first concatenating the labels of the edges $e_1,...,e_k$ and then removing all $1$'s, as $1$ is identified with the empty word. Therefore, the label of a path in a van Kampen diagram is a word over $\mathcal{A}$. If $w$ is a word over $\mathcal{A}$, then the notation $Lab(p)\equiv w$ will indicate a letter-by-letter equality between $Lab(p)$ and $w$.

\begin{remark}\label{decompose paths on diagrams}
Suppose that $p$ is a path in a van Kampen diagram over \eqref{pre0} with $Lab(p)\equiv w_1\cdot\cdot\cdot w_k$. Then we can decompose $p$ in the following way: Let $p_{w_1}$ be the maximal subpath of $p$ such that $p^-_{w_1}=p^-$ and $Lab(p_{w_1})\equiv w_1$. For $i=2,...,k$, let $p_{w_i}$ be the maximal subpath of $p$ such that $p^-_{w_i}=p^+_{w_{i-1}}$ and $Lab(p_{w_i})\equiv w_i$.
\end{remark}

Edges labeled by letters from $\mathcal{A}$ are called \textit{essential edges}, while edges labeled by the letter $1$ are called \textit{non-essential edges}. A \textit{face} of $\Delta$ is a $2$-cell of $\Delta$. Let $\Pi$ be a face of $\Delta$, the boundary of $\Pi$ is denoted as $\partial\Pi$. Likewise, the boundary of $\Delta$ is denoted by $\partial\Delta$. Note that if we choose a base point for $\partial\Pi$ (resp. $\partial\Delta$), then $\partial\Pi$ (resp. $\partial\Delta$) becomes a path in $\Delta$. For a word $w$ over $\mathcal{A}$, we use the notation $Lab(\partial\Pi)\equiv w$ (resp. $Lab(\partial \Delta)\equiv w$) to indicate that one can pick a base point to turn $\partial\Pi$ (resp. $\partial\Delta$) into a path $p$ so that $Lab(p)\equiv w$.

\begin{remark}\label{decompose loops on diagrams}
Suppose that $\Delta$ is a diagram with $Lab(\partial\Delta)\equiv w_1\cdot\cdot\cdot w_k$. Then we can decompose $\partial\Delta$ in the following way: Let $p_b$ be vertex of $\partial\Delta$ such that when we use $p_b$ as the base point of $\partial\Delta$, we can turn $\partial\Delta$ into a path $p$ with $Lab(p)\equiv w_1\cdot\cdot\cdot w_k$. And then we use Remark \ref{decompose paths on diagrams} to decompose $p$ and thus decompose $\partial\Delta$.
\end{remark}

Consider the following additional assumption on van Kampen diagrams:
\begin{enumerate}
\item[(c)] For every face $\Pi$ of a van Kampen diagram $\Delta$ over the presentation $\eqref{pre0}$, at least one of the following conditions (c1) and (c2) holds:
\item[(c1)] $Lab(\partial\Pi)$ is equal to an element of $\mathcal{R}$.
\item[(c2)] $\partial\Pi$ either consists entirely of non-essential edges or consists of exactly two essential edges with mutually inverse labels (in addition to non-essential edges).
\end{enumerate}

A face satisfying (c$_2$) is called a \textit{non-essential face}. All other faces are called \textit{essential faces}. The process of adding non-essential faces to a van Kampen diagram is called a \textit{refinement}. Figure \ref{0refine} illustrates a refinement on a van Kampen diagram, where the unlabeled edges are labeled by $1$. The interested readers are referred to \cite{ol2012geometry} for a formal discussion. By using refinements, we can ensure that
\begin{enumerate}
\item[(d)] Every face is homeomorphic to a disc, i.e., its boundary has no self-intersection.
\end{enumerate}

\begin{figure}
{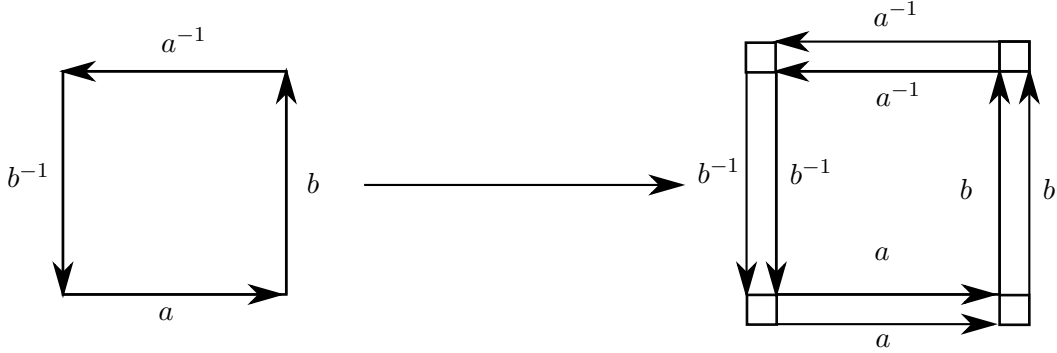}
\caption{A refinement of a van Kampen diagram over the presentation $G=\langle a,b\mid aba^{-1}b^{-1}=1\rangle$}\label{0refine}
\end{figure}

\begin{assumption}
In the sequel, the above assumptions (c) and (d) will be imposed on van Kampen diagrams.
\end{assumption}

The well-known van Kampen lemma states that a word $w$ over $\mathcal{A}$ represents $1$ in $G$ if and only if there is a van Kampen diagram $\Delta$ over \eqref{pre0} such that $\Delta$ is homeomorphic to a disc (such diagrams are called \textit{disk diagrams}), and that $Lab(\partial\Delta)\equiv w$.

\begin{remark}\label{mapofvk}
If a van Kampen diagram $\Delta$ is homeomorphic to a disc, and $O$ is a vertex of $\Delta$, then there exists a unique continuous map $\mu$ from the $1$-skeleton of $\Delta$ to $\Gamma(G,\mathcal{A})$ sending $O$ to the identity vertex, preserving the labels of the essential edges and collapsing non-essential edges to points.
\end{remark}

\subsection{Hyperbolically embedded subgroups and group theoretic Dehn fillings}\label{subsec.wh}
Let $G$ be a group, let $\{H_\lambda\}_{\lambda\in\Lambda}$ be a family of subgroups of $G$, let $X$ be a subset of $G$ such that $G$ is generated by $X$ together with the union of all $H_{\lambda},\lambda\in\Lambda$, and let $\mathcal{H}=\bigsqcup_{\lambda\in\Lambda}H_{\lambda}$. Consider the Cayley graph $\Gamma(G,X\sqcup \mathcal{H})$.

\begin{remark}
It is possible that $X$ and $H_{\lambda},\lambda\in\Lambda$, as subsets of $G$, have non-empty intersections with each other. As a consequence, several letters of $X\sqcup \mathcal{H}$ might represent the same element of $G$. If this is the case, the Cayley graph $\Gamma(G,X\sqcup \mathcal{H})$ will have multiple edges corresponding to those letters.
\end{remark}

\noindent
\textbf{Notation.}
Let $w$ be a word over the alphabet $X\sqcup \mathcal{H}$. The inverse of $w$ is denoted by $w^{-1}$. The \textit{length} of $w$, denoted as $\|w\|$, is the number of letters in $w$. We identify $w$ with the element of $G$ represented by $w$. Thus, if $S$ is a subset of $G$, then we write $w\in S$ to indicate that $w$ represents an element of $S$.

There are two types of equalities for words over $X\sqcup \mathcal{H}$. Given two words $u$ and $v$ over $X\sqcup \mathcal{H}$, the notation $u\equiv v$ indicates the letter-by-letter equality between $u$ and $v$ and the notation $u=_G v$ indicates that $u$ and $v$ represent the same element of $G$.

If $g$ is an element of $G$, then $|g|$ denotes the length of the shortest word in $X\sqcup \mathcal{H}$ representing $g$.



Note that, for each $\lambda\in\Lambda$, the Cayley graph $\Gamma(H_{\lambda},H_{\lambda})$ can be identified as the complete subgraph of $\Gamma(G,X\sqcup \mathcal{H})$ whose vertex set is $H_{\lambda}$, and edges are the ones labeled by letters from $H_{\lambda}$.

\begin{definition}
Fix $\lambda\in\Lambda$. A (combinatorial) path $p$ in $\Gamma(G,X\sqcup \mathcal{H})$ between vertices of $\Gamma(H_{\lambda},H_{\lambda})$ is called $H_{\lambda}$\textit{-admissible} if it does not contain any edge of $\Gamma(H_{\lambda},H_{\lambda})$. Note that a $H_{\lambda}$-admissible path $p$ is allowed to pass through vertices of $\Gamma(H_{\lambda},H_{\lambda})$. For every pair of elements $h,k\in H_{\lambda}$, let $\widehat{d}_{\lambda}(h,k)\in[0,+\infty]$ be the length of a shortest $H_{\lambda}$-admissible path connecting $h,k$. If no such path exists, set $\widehat{d}_{\lambda}(h,k)=+\infty$. The laws of summation on $[0,+\infty)$ extend naturally to $[0,+\infty]$ and it is easy to verify that $\widehat{d}_{\lambda}:H_{\lambda}\times H_{\lambda}\rightarrow [0,+\infty]$ defines a metric on $\Gamma(H_{\lambda},H_{\lambda})$ called the \textit{relative metric on} $\Gamma(H_{\lambda},H_{\lambda})$ \textit{with respect to} $X$.
\end{definition}

\begin{remark}
Let $p$ be a path in $\Gamma(G,X\sqcup \mathcal{H})$ with $Lab(p)\equiv h\in H_{\lambda}$, for some $\lambda\in\Lambda$. For simplicity, we denote $\widehat{d}_{\lambda}(1,h)$ by $\widehat{\ell}_{\lambda}(p)$.
\end{remark}

\begin{definition}\label{hyperbolically embedded}
Let $G$ be a group, let $\{H_\lambda\}_{\lambda\in\Lambda}$ be a family of subgroups of $G$, let $X$ be a subset of $G$, and let $\mathcal{H}=\bigsqcup_{\lambda\in\Lambda}H_{\lambda}$. We say that $\{H_\lambda\}_{\lambda\in\Lambda}$ \textit{weakly hyperbolically embeds into} $(G,X)$ (denoted as $\{H_\lambda\}_{\lambda\in\Lambda}\hookrightarrow_{wh}(G,X)$) if $G$ is generated by the set $X$ together with the union of all $H_{\lambda},\lambda\in\Lambda$, and the Cayley graph $\Gamma(G,X\sqcup \mathcal{H})$ is a Gromov hyperbolic space.

If the collection $\{H_\lambda\}_{\lambda\in\Lambda}\hookrightarrow_{wh}(G,X)$ and for each $\lambda\in\Lambda$, the metric space $(H_{\lambda},\widehat{d}_{\lambda})$ is proper, i.e., every ball of finite radius contains only finitely many elements, then we say that $\{H_\lambda\}_{\lambda\in\Lambda}$ \textit{hyperbolically embeds into} $(G,X)$ (denoted as $\{H_\lambda\}_{\lambda\in\Lambda}\hookrightarrow_h (G,X)$).

Further, the collection $\{H_\lambda\}_{\lambda\in\Lambda}$ \textit{hyperbolically embeds into} $G$, denoted as $\{H_\lambda\}_{\lambda\in\Lambda}\hookrightarrow_h G$, if there exists some subset $X\subset G$ such that $\{H_\lambda\}_{\lambda\in\Lambda}\hookrightarrow_h(G,X)$.
\end{definition}

\begin{remark}
Note that if the family $\{H_\lambda\}_{\lambda\in\Lambda}\hookrightarrow_{wh}(G,X)$ for some subset $X\subset G$ and $Y=X\cup X^{-1}$, then we also have $\{H_\lambda\}_{\lambda\in\Lambda}\hookrightarrow_{wh}(G,Y)$. In the sequel, we will always assume that the relative generating set $X$ is symmetric, i.e., $X=X^{-1}$.
\end{remark}

\begin{definition}
Suppose $G$ is a group with a family of subgroups $\{H_{\lambda}\}_{\lambda\in\Lambda}\hookrightarrow_{wh}(G,X)$ for some subset $X\subset G$. For $\lambda\in\Lambda$, let $\widehat{d}_{\lambda}$ be the relative metric on $\Gamma(H_{\lambda},H_{\lambda})$ with respect to $X$. We say that a property $P$ \textit{holds for all sufficiently deep Dehn fillings of} $\{H_{\lambda}\}_{\lambda\in\Lambda}$ (or \textit{for all sufficiently deep} $N_{\lambda}\lhd H_{\lambda},\lambda\in\Lambda,$) if there exists a number $C>0$ such that if $N_{\lambda}\lhd H_{\lambda}$ and $\widehat{d}_{\lambda}(1,n)>C$ for all $n\in N_{\lambda}\setminus\{1\},\lambda\in\Lambda$, then $P$ holds.
\end{definition}

One remarkable property of weakly hyperbolically embedded subgroups is the following group theoretic Dehn filling theorem.

\begin{theorem}[{\cite[Theorem 7.15]{dahmani2017hyperbolically}}]\label{Dehn filling}
Let $G$ be a group with a family of subgroups $\{H_\lambda\}_{\lambda\in\Lambda}\hookrightarrow_{wh}(G,X)$ for some subset $X\subset G$. Then for all sufficiently deep $N_{\lambda}\lhd H_{\lambda},\lambda\in\Lambda$, we have:
\begin{enumerate}
\item[(a)] For each $\lambda\in\Lambda$, the natural homomorphism $i_{\lambda}:H_{\lambda}/N_{\lambda}\rightarrow G/\ll \mathcal{N} \rr$ is injective (i.e., $H_{\lambda}\cap \ll \mathcal{N} \rr=N_{\lambda}\left.\right)$, where $\mathcal{N}=\bigcup_{\lambda\in\Lambda}N_{\lambda}$.
\item[(b)] $\{i_{\lambda}(H_{\lambda}/N_{\lambda})\}_{\lambda\in\Lambda}\hookrightarrow_{wh}G/\ll \mathcal{N} \rr$.
\item[(c)] There exist subsets $T_{\lambda}\subset G,\lambda\in\Lambda$, such that $\ll \mathcal{N} \rr = \prod^{\ast}_{\lambda\in\Lambda,t\in T_{\lambda}} N^t_{\lambda}$.
\end{enumerate}
\end{theorem}

As mentioned in the introduction, we are going to improve part (c) of the above theorem. To simplify statements, we introduce the following terminologies.

\begin{definition}
Let $G$ be a group with a subgroup $H$. The \textit{collection of left transversals of} $H$ \textit{in} $G$ is denoted by $LT(H,G)$, i.e., every element of $LT(H,G)$ is a left transversal of $H$ in $G$. 
\end{definition}

\begin{definition}\label{CLp}
Let $G$ be a group, let $\{H_{\lambda}\}_{\lambda\in\Lambda}$ be a family of subgroups of $G$, and let $N_{\lambda}$ be a normal subgroup of $H_{\lambda}$ for every $\lambda\in\Lambda$. We say that the triple $(G,\{H_{\lambda}\}_{\lambda\in\Lambda},\{N_{\lambda}\}_{\lambda\in\Lambda})$ has the \textit{Cohen-Lyndon property} if there exists a left transversal $T_{\lambda}\in LT(H_{\lambda}\ll \mathcal{N} \rr,G)$ for every $\lambda\in\Lambda$ such that $\ll \mathcal{N} \rr = \prod^{\ast}_{\lambda\in\Lambda,t\in T_{\lambda}} N^t_{\lambda}$, where $\mathcal{N}=\bigcup_{\lambda\in\Lambda}N_{\lambda}$.
\end{definition}

\subsection{Isolated components}\label{subsec.ic}
Let us assume, until the end of Section \ref{proof}, that $G$ is a group with a family of subgroups $\{H_\lambda\}_{\lambda\in\Lambda}\hookrightarrow_{wh}(G,X)$ for some symmetric subset $X\subset G$. For each $\lambda\in\Lambda$, let $\widehat{d}_{\lambda}$ be the relative metric on $\Gamma(H_{\lambda},H_{\lambda})$ with respect to $X$, and let $\mathcal{H}=\bigsqcup_{\lambda\in\Lambda}H_{\lambda}$. The following terminology goes back to \cite{osin2006relatively}.

\begin{definition}\label{connect}
Let $p$ be a path in $\Gamma(G,X\sqcup \mathcal{H})$. Fix $\lambda\in\Lambda$. An $H_{\lambda}$\textit{-subpath} $q$ of $p$ is a nontrivial subpath of $p$ labeled by a word over the alphabet $H_{\lambda}$ (if $p$ is a cycle, we allow $q$ to be a subpath of some cyclic shift of $p$). An $H_{\lambda}$-subpath $q$ of $p$ is called an $H_{\lambda}$-\textit{component} if $q$ is not properly contained in any other $H_{\lambda}$-subpath. Two $H_{\lambda}$-components $q_1,q_2$ of $p$ are called \textit{connected} if there exists a path $c$ in $\Gamma(G,X\sqcup \mathcal{H})$ such that $c$ connects a vertex of $q_1$ to a vertex of $q_2$, and that $Lab(c)$ is a letter from $H_{\lambda}$. An $H_{\lambda}$-component $q$ of $p$ is called \textit{isolated} if it is not connected to any other $H_{\lambda}$-component of $p$.
\end{definition}

The key property of isolated components is that, in a \textit{geodesic polygon} (i.e., a polygon in $\Gamma(G,X\sqcup \mathcal{H})$ with geodesic sides) $p$, the total $\widehat{\ell}$-length of isolated components is uniformly bounded above by a linear function of the number of sides. The following result is proved in \cite[Proposition 4.14]{dahmani2017hyperbolically}, which is a straightforward generalization of \cite[Proposition 3.2]{osin2007peripheral}.

\begin{lemma}[Dahmani-Guirardel-Osin]\label{length}
There exists a positive number $D$ satisfying the following property: Let $p$ be an $n$-gon in $\Gamma(G,X\sqcup \mathcal{H})$ with $(2,0)$-quasi-geodesic sides $p_1,...,p_n$ and let $I$ be a subset of the set of sides of $p$ such that every side $p_i\in I$ is an isolated $H_{\lambda_i}$-component of $p$ for some $\lambda_i\in\Lambda$. Then
$$\sum_{p_i\in I}\widehat{\ell}_{\lambda_i}(p_i)\leqslant Dn.$$
\end{lemma}

\begin{remark}\label{constant}
Theorem \ref{Dehn filling} asserts the existence of a constant $C$ such that if $\widehat{d}_{\lambda}(1,n)\geqslant C$ for every $n\in N_{\lambda}\setminus\{1\}$ and $\lambda\in\Lambda$, then $H_{\lambda}\cap \ll\mathcal{N}\rr=N_{\lambda}$ for all $\lambda\in\Lambda$. In fact, one can let $C=4D$, where $D$ is the constant provided by Lemma \ref{length} (see \cite{dahmani2017hyperbolically}).

We would like to clarify here that in the published version of this paper, Lemma \ref{length} considers $n$-gon with geodesic sides instead of $(2,0)$-quasi-geodesic sides. However, our proof uses \cite[Theorem 7.15]{dahmani2017hyperbolically}, which requires the constant $D$ to satisfy the current version of Lemma \ref{length}.
\end{remark}

\subsection{Diagram surgery}\label{sec.surgery}
The diagram surgery surveyed in this section was first introduced by Osin in \cite{osin2007peripheral}, where he proved a group theoretic Dehn filling theorem for relatively hyperbolic groups. Later, Dahmani et al. generalized this technique to deal with weakly hyperbolically embedded subgroups \cite{dahmani2017hyperbolically}.

Consider a symmetric set $\mathcal{R}$ of words over the alphabet $X\sqcup \mathcal{H}$ such that $G$ has the presentation
\begin{equation}\label{pre1}
G=\langle X\sqcup \mathcal{H}\mid\mathcal{R}\rangle,
\end{equation}
and that for all $\lambda\in\Lambda$, $\mathcal{R}$ contains all words over the alphabet $H_{\lambda}$ which represent the identity.

Suppose that $N_{\lambda}$ is a normal subgroup of $H_{\lambda}$ for each $\lambda\in\Lambda$. Denote the union of $N_{\lambda},\lambda\in\Lambda$, by $\mathcal{N}$. Killing $\ll \mathcal{N} \rr$ in $G$ is equivalent to adding, to $\mathcal{R}$, all words over $H_{\lambda}$ which represent elements of $N_{\lambda}$, for all $\lambda\in\Lambda$, to form a new presentation
\begin{equation}\label{pre}
Q=G/\ll\mathcal{N}\rr=\langle X\sqcup \mathcal{H},\mathcal{R}\cup \mathcal{S}\rangle,
\end{equation}
where $\mathcal{S}=\bigcup_{\lambda\in\lambda}\mathcal{S}_{\lambda}$ and $\mathcal{S}_{\lambda}$ consists of all words over $H_{\lambda}$ representing elements of $N_{\lambda}$ in $G$.

In the sequel, let $\mathcal{D}$ be the set of all van Kampen diagrams $\Delta$ over \eqref{pre} satisfying the following.

\begin{enumerate}
\item[(D1)] Topologically $\Delta$ is a disc with $k\geqslant 0$ holes. The boundary of $\Delta$ can be decomposed as $\partial \Delta=\partial_{ext}\Delta\cup\partial_{int}\Delta$, where $\partial_{ext}\Delta$ is the boundary of the disc, and $\partial_{int}\Delta$ consists of disjoint cycles (connected components) $c_1,...,c_k$ that bound the holes.
\item[(D2)] For $i=1,...,k$, $c_i$ is labeled by a word from $\mathcal{S}$.
\item[(D3)] Each diagram $\Delta$ is equipped with a \textit{cut system} that is a collection $T=\{t_1,...,t_k\}$ of disjoint paths (\textit{cuts}) $t_1,...,t_k$ in $\Delta$ without self-intersections such that, for $i=1,...,k$, the two endpoints of $t_i$ belong to $\partial\Delta$, and that after cutting $\Delta$ along $t_i$ for all $i = 1,...,k$, one gets a disc van Kampen diagram $\widetilde{\Delta}$ over \eqref{pre1}.
\end{enumerate}

See Figure \ref{findcut} for an illustration of a diagram in $\mathcal{D}$.



\begin{lemma}
A word $w$ over $X\sqcup \mathcal{H}$ represents $1$ in $Q$ if and only if there is a diagram $\Delta\in \mathcal{D}$ such that $Lab(\partial_{ext}\Delta)\equiv w$.
\end{lemma}

\begin{proof}
Let $w$ be a word over $X\sqcup \mathcal{H}$. If there is a diagram $\Delta\in \mathcal{D}$ such that $\partial_{ext}\Delta \equiv w$, by filling the holes of $\Delta$ with faces whose boundaries are labeled by words from $\mathcal{S}$, one creates a disc van Kampen diagram over \eqref{pre1}, whose boundary is labeled by $w$. Conversely, if $w$ represents $1$ in $Q$, then there exists a disc van Kampen diagram $\overline{\Delta}$ over \eqref{pre1} with $Lab(\partial\overline{\Delta})\equiv w$. By removing all faces of $\overline{\Delta}$ labeled by words from $\mathcal{S}$, we obtain a diagram $\Delta'$ satisfying (D1) and (D2). To produce a cut system, choose a vertex $O$ in $\partial_{ext}\Delta'$. Connect $O$ with each component of $\partial_{int}\Delta'$ by a path so that these paths do not cross each other (although they do intersect each other). By passing to a refinement of $\Delta'$, one can separate these paths so that they no longer intersect each other and thus creates a diagram $\Delta$ satisfying (D1), (D2), and (D3) with $Lab(\partial_{ext}\Delta)\equiv w$.
\end{proof}

Figure \ref{findcut} illustrates the last step of the above proof. The left half shows the diagram $\Delta'$ where a dash path and a dot-dash path connect $O$ with two components of $\partial_{int}\Delta'$. By thickening these paths with a refinement, we obtain the right half. The regions shaded by horizontal and vertical lines consist of non-essential faces, while the paths from $O_1$ to $U_1$ and from $O_3$ to $V_2$ form a cut system.

\begin{figure}
{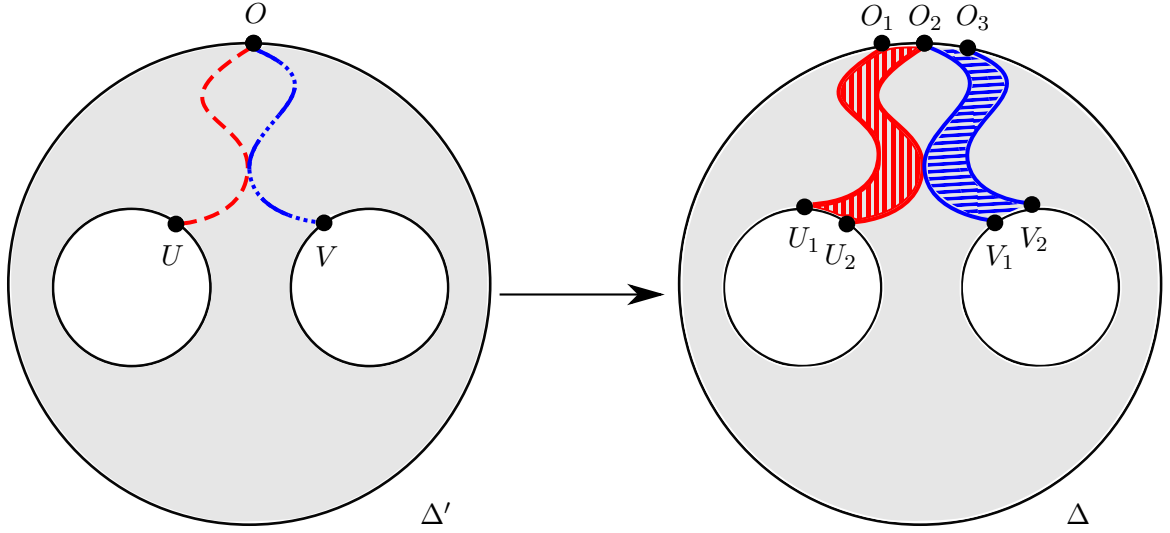}
\caption{How to produce a cut system}\label{findcut}
\end{figure}

Let $\Delta$ be a diagram in $\mathcal{D}$ and let $\widetilde{\Delta}$ be the disc van Kampen diagram resulting from cutting $\Delta$ along its set of cuts. Define $\kappa:\widetilde{\Delta}\rightarrow\Delta$ to be the map that ``sews'' the cuts. Fix an arbitrary vertex $O$ in $\widetilde{\Delta}$ and let $\mu$ be a map sending the $1$-skeleton of $\Delta$ to $\Gamma(G,X\sqcup \mathcal{H})$, as described by Remark \ref{mapofvk}.

\begin{definition}
Let $\Delta_1$ and $\Delta_2$ be two diagrams of $\mathcal{D}$ and let $\Gamma_1$ (resp. $\Gamma_2\left.\right)$ be the subgraph of the $1$-skeleton of $\Delta_1$ (resp. $\Delta_2\left.\right)$ consisting of $\partial\Delta_1$ (resp. $\partial\Delta_2\left.\right)$ and all cuts of $\Delta_1$ (resp. $\Delta_2\left.\right)$. We say that $\Delta_1$ and $\Delta_2$ are \textit{equivalent} if there exists a graph isomorphism $\Gamma_1\rightarrow\Gamma_2$ which preserves labels and orientations of edges, and maps the cuts and boundary of $\Delta_1$ to the cuts and boundary of $\Delta_2$, respectively.
\end{definition}

The following Lemmas \ref{adding path} and \ref{cut} are results from \cite{dahmani2017hyperbolically}, which are straightforward generalizations of results of \cite{osin2007peripheral}. Note that the authors of \cite{dahmani2017hyperbolically} assume that the presentation \eqref{pre1} has a linear relative isoperimetric function, but this assumption is not used in the proofs of those lemmas.

\begin{lemma}[{\cite[Lemma 7.11]{dahmani2017hyperbolically}}]\label{adding path}
Let $a,b$ be two vertices on $\partial\Delta$ and let $\widetilde{a},\widetilde{b}$ be two vertices on $\partial\widetilde{\Delta}$ such that $\kappa(\widetilde{a})=a,\kappa(\widetilde{b})=b$. Then for any path $p$ in $\Gamma(G,X\sqcup \mathcal{H})$ connecting $\mu(\widetilde{a})$ to $\mu(\widetilde{b})$, there is a diagram $\Delta_1\in\mathcal{D}$ with the following properties:

\begin{enumerate}
\item[(a)] $\Delta$ and $\Delta_1$ are equivalent.
\item[(b)] There is a path $q$ in $\Delta_1$ without self-intersections such that (1) $q$ connects $a$ and $b$, (2) $q$ has no common vertices with the cuts of $\Delta_1$ except possibly for $a,b$, and (3) $Lab(q)\equiv Lab(p)$.
\end{enumerate}
\end{lemma}

\begin{definition}
Fix $\lambda\in\Lambda$. An $H_{\lambda}$\textit{-subpath} in $\partial\Delta$ (resp. $\partial\widetilde{\Delta}$) for some $\Delta\in \mathcal{D}$ is a path labeled by a nontrivial word over $H_{\lambda}$. An $H_{\lambda}$-subpath $p$ of $\partial\Delta$ (resp. $\partial\widetilde{\Delta}$) is called an $H_{\lambda}$\textit{-component} if $p$ is not properly contained in any other $H_{\lambda}$-subpath. Two $H_{\lambda}$-components $p,q$ of $\partial\Delta$ are \textit{connected} if there exist $H_{\lambda}$-components $a,b$ in $\partial\widetilde{\Delta}$ such that $\kappa(a)$ (resp. $\kappa(b)$) is a subpath of $p$ (resp. $q$), and that $\mu(a),\mu(b)$ are connected in $\Gamma(G,X\sqcup \mathcal{H})$ (in the sense of Definition \ref{connect}).
\end{definition}

\begin{remark}
The definitions of $H_{\lambda}$-subpaths, $H_{\lambda}$-components, and connected $H_{\lambda}$-components in $\partial\Delta$ for a van Kampen diagram $\Delta\in\mathcal{D}$ or $\partial\widetilde{\Delta}$ do not depend on the pre-chosen vertex $O$.
\end{remark}

\begin{definition}\label{type of diagrams}
The \textit{type} of $\Delta$ is defined by the formula
$$\tau(\Delta)=(k,\sum_{i=1}^k\|Lab(t_i)\|),$$
where $k$ is the number of holes in $\Delta$ and $t_1,...,t_k$ are the cuts. We order the types of diagrams in $\mathcal{D}$ lexicographically: $(k_1,\ell_1)<(k_2,\ell_2)$ if and only if either $k_1<k_2$ or $k_1=k_2$ and $\ell_1<\ell_2$.
\end{definition}

\begin{definition}
A word $w$ over $X\sqcup\mathcal{H}$ is called \textit{geodesic} if it labels a geodesic in $\Gamma(G,X\sqcup\mathcal{H})$.
\end{definition}

\begin{definition}
For any word $w$ over $X\sqcup\mathcal{H}$, let $\mathcal{D}(w)$ be the set of diagrams $\Delta\in\mathcal{D}$ such that $Lab(\partial_{ext}\Delta)\equiv w$.
\end{definition}

\begin{lemma}[{\cite[Lemma 7.17]{dahmani2017hyperbolically} (see also \cite[Lemma 5.2]{osin2007peripheral})}]\label{cut}
Suppose that for every $\lambda\in\Lambda$ and $n\in N_{\lambda}\setminus\{1\}$, we have $\widehat{d}_{\lambda}(1,n)>4D$, where $D$ is the constant given by Lemma \ref{length}. Let $w$ be a geodesic word over $X\sqcup \mathcal{H}$ representing $1$ in $Q$, and let $\Delta$ be a diagram in $\mathcal{D}(w)$ of minimal type. Then there exists $\lambda\in\Lambda$ and a connected component $c$ of $\partial_{int}\Delta$ such that $c$ is connected to an $H_{\lambda}$-component of $\partial_{ext}\Delta$.
\end{lemma}

\section{Construction of the transversals}\label{sec.transversal}
The proof of Theorem \ref{main} relies on constructing a particular left transversal $T_{\lambda}\in LT(H_{\lambda}\ll\mathcal{N}\rr,G)$ for each $\lambda\in\Lambda$. It is convenient to construct a collection $\{T_{\lambda}\}_{\lambda\in\Lambda}$ of sets of words over $X\sqcup \mathcal{H}$ satisfying the following properties (P1) through (P3), and think of $T_{\lambda}$ as a transversal in $LT(H_{\lambda}\ll\mathcal{N}\rr,G)$ (identifying words over $X\sqcup\mathcal{H}$ and the elements of $G$ represented by those words) for $\lambda\in\Lambda$. Recall that $\|w\|$ is the length of $w$ for a word $w$ over $X\sqcup H$, and that $|g|$ denotes the length of a geodesic word over $X\sqcup\mathcal{H}$ representing an element $g\in G$.

\begin{enumerate}
\item[(P1)] [$\{T_{\lambda}\}_{\lambda\in\Lambda}$ is transversal] For each $\lambda\in\Lambda$, $T_{\lambda}\in LT(H_{\lambda}\ll\mathcal{N}\rr,G)$.
\item[(P2)] [$\{T_{\lambda}\}_{\lambda\in\Lambda}$ is geodesic] If $w\in T_{\lambda}$ for some $\lambda\in\Lambda$, and $gH_{\lambda}\ll\mathcal{N}\rr=wH_{\lambda}\ll\mathcal{N}\rr$ for some $g\in G$, then $\|w\|\leqslant |g|$. This implies that, for all $\lambda\in\Lambda$, every $w\in T_{\lambda}$ is a geodesic word over $X\sqcup\mathcal{H}$.
\item[(P3)] [$\{T_{\lambda}\}_{\lambda\in\Lambda}$ is prefix closed] Let $\lambda,\mu\in\Lambda$. If a word $w\in T_{\lambda}$ can be decomposed as $w\equiv uhv$ with $h\in H_{\mu}\setminus\{1\}$ ($u,v$ are allowed to be empty words), then $u\in T_{\mu}$ and $\widehat{d}_{\mu}(1,h)\leqslant \widehat{d}_{\mu}(1,h')$ for all $h'\in hN_{\mu}$.
\end{enumerate}

\begin{lemma}\label{transversal}
There exists a collection $\{T_{\lambda}\}_{\lambda\in\Lambda}$ satisfying (P1), (P2), and (P3).
\end{lemma}

\begin{proof}
Let $\mathcal{W}$ be the poset of collections $\{W_{\lambda}\}_{\lambda\in\Lambda}$ of words satisfying (P2) and (P3), while instead of (P1), we only demand that the words of $W_{\lambda}$ represent a subset of a transversal in $LT(H_{\lambda}\ll\mathcal{N}\rr,G)$ for every $\lambda\in\Lambda$. We order $\mathcal{W}$ by index-wise inclusion, i.e., $\{U_{\lambda}\}_{\lambda\in\Lambda}$ is less than $\{V_{\lambda}\}_{\lambda\in\Lambda}$ if and only if $U_{\lambda}\subset V_{\lambda}$ for every $\lambda\in\Lambda$. $\mathcal{W}$ is non-empty because the collection $\{W_{\lambda}\}_{\lambda\in\Lambda}$ with each $W_{\lambda}$ consisting of only the empty word is a member of $\mathcal{W}$. Moreover, the union of any chain of $\mathcal{W}$ is again a member of $\mathcal{W}$. Therefore, Zorn's lemma implies that $\mathcal{W}$ has a maximal member $\{T_{\lambda}\}_{\lambda\in\Lambda}$. Suppose that $\{T_{\lambda}\}_{\lambda\in\Lambda}$ does not satisfy (P1), i.e., there exist $\lambda_0\in\Lambda$ and $g\in G$ such that no element of the coset $gH_{\lambda_0}M$ is represented by a word in $T_{\lambda_0}$. Without loss of generality, let us assume that if $g'$ is an element of $G$ such that $|g'|<|g|$, then for each $\lambda\in\Lambda$, $g'H_{\lambda}\ll \mathcal{N} \rr\cap T_{\lambda}\neq \emptyset$.

Let $w$ be a geodesic word over $X\sqcup \mathcal{H}$ representing $g$. Consider the collection $\{U_{\lambda}\}_{\lambda\in\Lambda}$ constructed as follows. For every $\lambda\in\Lambda\setminus\{\lambda_0\}$, let $U_{\lambda}=T_{\lambda}$, and construct $U_{\lambda_0}$ by the following manner: If $w$ contains no letter from $\mathcal{H}$, let $U_{\lambda_0}=T_{\lambda_0}\cup\{w\}$. If $w$ contains at least one letter from $\mathcal{H}$, then $w$ can be decomposed as $w\equiv uhv$ such that $h\in H_{\lambda}\setminus\{1\}$ for some $\lambda\in\Lambda$ and $v$ contains no letter from $\mathcal{H}$ ($u,v$ are allowed to be empty words). As $\|u\|<\|w\|=|g|$, there exists a word $u'\in T_{\lambda}$ such that $u'\in uH_{\lambda}\ll \mathcal{N} \rr$. Let $h'$ be an element of $H_{\lambda}$ such that $u\ll \mathcal{N} \rr=u'h'\ll \mathcal{N} \rr$ and let $h''$ be an element of $H_{\lambda}$ such that (a) $h''N_{\lambda}=h'hN_{\lambda}$ and (b) if $k\in h''N_{\lambda}$, then $\widehat{d}_{\lambda}(1,h'')\leqslant\widehat{d}_{\lambda}(1,k)$. Set $U_{\lambda_0}=T_{\lambda_0}\cup\{u'h''v\}$.

It is straight-forward to verify that $\{U_{\lambda}\}_{\lambda\in\Lambda}$ is an element of $\mathcal{W}$. There is a word in $U_{\lambda_0}$ representing an element in $gH_{\lambda_0}\ll \mathcal{N} \rr$, while $T_{\lambda_0}$ has no such words. It follows that $\{U_{\lambda}\}_{\lambda\in\Lambda}$ is strictly greater than $\{T_{\lambda}\}_{\lambda\in\Lambda}$, contradicting the choice of $\{T_{\lambda}\}_{\lambda\in\Lambda}$.
\end{proof}

\section{Main theorem and its proof}\label{proof}
The following Theorem \ref{main} is the main theorem of this paper. We first use it to prove Theorem \ref{simmain}, and then we prove Theorem \ref{main}.

\begin{theorem}\label{main}
Let $G$ be a group with a family of subgroups $\{H_{\lambda}\}_{\lambda\in\Lambda}\hookrightarrow_{wh}(G,X)$ for some $X\subset G$. Then the Cohen-Lyndon property holds for all sufficiently deep Dehn fillings of $\{H_{\lambda}\}_{\lambda\in\Lambda}$.
\end{theorem}

\begin{proof}[Proof of Theorem \ref{simmain}]
By assumption, $H\hookrightarrow_h (G,X)$ for some subset $X\subset G$. Let $\widehat{d}$ be the relative metric on $\Gamma(H,H)$ with respect to $X$. Theorem \ref{main} provides a constant $C$ such that if $N\lhd H$ and $\widehat{d}(n)>C$ for all $n\in N\setminus\{1\}$, then $(G,H,N)$ possesses the Cohen-Lyndon property. As $H\hookrightarrow_h (G,X)$, $\widehat{d}$ is locally finite. In particular, 
$$\mathcal{F}=\{h\in H\setminus\{1\}\mid \widehat{d}(h)\leqslant C\}$$
is a finite set. The desired result follows by noting that if $N\lhd H$ and $N\cap \mathcal{F}=\emptyset$, then $(G,H,N)$ possesses the Cohen-Lyndon property.
\end{proof}

Let us prove Theorem \ref{main}. Recall that Lemma \ref{length} provides a number $D>0$ to estimate the total length of isolated components in a (2,0)-quasi-geodesic polygon, and that Theorem \ref{Dehn filling} and Remark \ref{constant} implies that if $\widehat{d}_{\lambda}(1,n)\geqslant 4D$ for every $n\in N_{\lambda}\setminus\{1\}$ and $\lambda\in\Lambda$, then $H_{\lambda}\cap \ll\mathcal{N}\rr=N_{\lambda}$ for all $\lambda\in\Lambda$. We assume the following condition.

\begin{enumerate}
\item[(24D)] $\widehat{d}_{\lambda}(1,n)>24D$ for all $n\in N_{\lambda}\setminus\{1\}$ and $\lambda\in\Lambda$.
\end{enumerate}

We prove that (24D) implies the Cohen-Lyndon property of $(G,\{H_{\lambda}\}_{\lambda\in\Lambda},\{N_{\lambda}\}_{\lambda\in\Lambda})$. Let $\{T_{\lambda}\}_{\lambda\in\Lambda}$ be a collection of words over $X\sqcup\mathcal{H}$ satisfying (P1), (P2), and (P3) (by Lemma \ref{transversal}, such a collection exists) and think of each $T_{\lambda}$ as a left transversal in $LT(H_{\lambda}\ll\mathcal{N}\rr,G)$. For every $\lambda\in\Lambda$, we extend $T_{\lambda}$ to a set $T^{ex}_{\lambda}$. Roughly speaking, $T^{ex}_{\lambda}$ is the set of words obtained from $T_{\lambda}$ by replacing letters from $H_{\lambda}$ with other letters from the same coset of $N_{\lambda}$ in $H_{\lambda}$.

\begin{definition}\label{extended T}
For every $\lambda\in\Lambda$, let $T^{ex}_{\lambda}$ be the set of words with the following property: Every word $w\in T^{ex}_{\lambda}$ admits a decomposition $w\equiv w_1h_1\cdot\cdot\cdot w_kh_kw_{k+1}$ ($w_1,...,w_{k+1}$ are allowed to be empty words) such that for every $i\in \{1,...,k\}$, there exists $\lambda_i\in\Lambda$ with the following properties.
\begin{enumerate}
\item[(a)] For $i=1,...,k$, $h_i$ is an element of $H_{\lambda_i}$ ($h_i$ is allowed to equal $1$).
\item[(b)] There exists an element $h'_i\in H_{\lambda_i}\setminus\{1\}$ such that $h'_iN_{\lambda_i}=h_iN_{\lambda_i}$ for $i=1,...,k$, and that the concatenation $w_1h'_1\cdot\cdot\cdot w_kh'_kw_{k+1}$ is a word in $T_{\lambda}$.
\end{enumerate}
\end{definition}

\begin{remark}
If $k=0$ in the above definition, conditions (a) and (b) will be satisfied trivially. Thus, $T_{\lambda}$ is a subset of $T^{ex}_{\lambda}$ for every $\lambda\in\Lambda$.
\end{remark}

\begin{definition}
Let $w$ be a word over $X\sqcup\mathcal{H}$ and let $\lambda\in\Lambda$. If $w\in T^{ex}_{\lambda}$, let $rank_{\lambda}(w)$ be the minimal number $k$ obtained from the decompositions $w\equiv w_1h_1\cdot\cdot\cdot w_kh_kw_{k+1}$ satisfying Definition \ref{extended T}. If $w\not\in T_{\lambda}$, let $rank_{\lambda}(w)=\infty$.

For every word $w$ over $X\sqcup\mathcal{H}$, the \textit{rank} of $w$, denoted as $rank(w)$, is the number $\min_{\lambda\in\Lambda}\{rank_{\lambda}(w)\}$.
\end{definition}

\begin{lemma}\label{u4}
Let $w$ be a word in $T^{ex}_{\lambda}$ for some $\lambda\in\Lambda$. Suppose that $w$ can be decomposed as $w\equiv uhv$ with $h\in H_{\mu}\setminus\{1\}$ for some $\mu\in\Lambda$. Let $h''$ be an element of $H_{\mu}$ such that $h''N_{\mu}=hN_{\mu}$. Then $uh''v\in T^{ex}_{\lambda}$.
\end{lemma}

\begin{proof}
Let $w\equiv w_1h_1\cdot\cdot\cdot w_kh_kw_{k+1}$ be a decomposition satisfying Definition \ref{extended T} and let $h'_1,...,h'_k$ be as in (b) of Definition \ref{extended T}.





Without loss of generality, we may assume that $h=h_i$ for some number $i\in \{1,...,k\}$. Then $uh''v$ can be decomposed as
$$uh''v\equiv w_1h_1\cdot\cdot\cdot w_{i-1}h_{i-1}w_ih''w_{i+1}h_{i+1}w_{i+2}h_{i+2}\cdot\cdot\cdot w_kh_kw_{k+1}.$$

By replacing $h_j$ with $h'_j$ for $j\neq i$ and $h''$ with $h'_i$, we obtain a word in $T_{\lambda}$ and thus $uh''v\in T^{ex}_{\lambda}$.
\end{proof}

\begin{lemma}\label{trunc}
Let $w$ be a word in $T^{ex}_{\lambda}$ for some $\lambda\in\Lambda$ with a decomposition $w\equiv w_1h_1\cdot\cdot\cdot w_kh_kw_{k+1}$ satisfying Definition \ref{extended T}. Then $w_1\in T_{\lambda_1}$.
\end{lemma}

\begin{proof}
Let $h'_1,...,h'_k$ be as in (b) of Definition \ref{extended T}. Note that the word $w_1h'_1\cdot\cdot\cdot w_kh'_kw_{k+1}$ can be decomposed as $$w_1h'_1...w_kh'_kw_{k+1}\equiv w_1h'_1(w_2h'_2\cdot\cdot\cdot w_kh'_kw_{k+1}).$$
By (P3), $w_1\in T_{\lambda_1}$.
\end{proof}



It will be shown that $\ll \mathcal{N} \rr = \prod^{\ast}_{\lambda\in\Lambda,t\in T_{\lambda}}N^t_{\lambda}$. For the moment, let
$$K=\langle N^t_{\lambda},t\in T_{\lambda},\lambda\in\Lambda\rangle\leqslant G.$$

\begin{lemma}\label{inK}
Let $w$ be a word in $\bigcup_{\lambda\in\Lambda}T^{ex}_{\lambda}$, and let $n$ be an element of $N_{\lambda_0}$ for some $\lambda_0\in\Lambda$. Then $wnw^{-1}\in K$.
\end{lemma}

\begin{proof}
Let $\mu$ be an element of $\Lambda$ with $rank(w)=rank_{\mu}(w)$. Thus, $w$ admits a decomposition $w\equiv w_1h_1\cdot\cdot\cdot w_kh_kw_{k+1}$ satisfying Definition \ref{extended T} with $k=rank(w)$. We perform induction on $rank(w)$. If $rank(w)=0$, then $w\in T_{\mu}$ and thus $wnw^{-1}\in K$.

Suppose that, for all $w'\in \bigcup_{\lambda\in\Lambda}T^{ex}_{\lambda}$ with $rank(w')<rank(w)$ and all $n'\in\bigcup_{\lambda\in\Lambda}N_{\lambda}$, we have $w'^{-1}n'w'\in K$. Let $h'_1,...,h'_k$ be as in (b) of Definition \ref{extended T}. Thus, there exists $n_1\in N_{\lambda_1}$ such that $n_1h'_1=h_1$ (note that $N_{\lambda_1}$ is a normal subgroup of $H_{\lambda_1}$). Notice that
$$w=_G (w_1n_1w_1^{-1})(w_1h'_1w_2h_2\cdot\cdot\cdot w_kh_kw_{k+1})$$
and thus
\begin{equation}\label{just use to refer in ink}
wnw^{-1}=_G (w_1n_1w_1^{-1})(w'nw'^{-1})(w_1n_1w_1^{-1})^{-1},
\end{equation}
where $w'\equiv w_1h'_1w_2h_2\cdot\cdot\cdot w_kh_kw_{k+1}$.

By replacing $h_j$ with $h'_j$ for $j=2,...,k$, we can turn $w'$ into a word in $T_{\mu}$. Thus, $w'\in T^{ex}_{\mu}$ and $rank(w')\leqslant k-1 < rank(w)$. It follows from the induction hypothesis that $w'n(w')^{-1}\in K$. By Lemma \ref{trunc}, $w_1\in T_{\lambda_1}$ and thus $w_1n_1w_1^{-1}\in K$. By \eqref{just use to refer in ink}, $wnw^{-1}$ represents a product of elements of $K$.
\end{proof}

For the next two lemmas, recall that $\|w\|$ denotes the length of a word $w$ over $X\sqcup\mathcal{H}$, and that $|g|$ denotes the length of a geodesic word over $X\sqcup \mathcal{H}$ representing an element $g\in G$.

\begin{lemma}\label{bbjs}
Let $\lambda$ be an element of $\Lambda$, $u$ be a word in $T^{ex}_{\lambda}$, $h$ be a letter of $H_{\lambda}\setminus\{1\}$, and $v$ be a word over $X\sqcup\mathcal{H}$ with $\|v\|=\|u\|$. Suppose that every element $m'\in \ll\mathcal{N}\rr$ with $|m'|<2\|u\|+1$ belongs to $K$. If the concatenation $uhv\in \ll\mathcal{N}\rr$, then $uhv\in K$.
\end{lemma}

\begin{proof}
If $uhv$ is not a geodesic word, the desired result will follow from the assumptions trivially. So let us assume that $uhv$ is geodesic. Consider a diagram $\Delta\in\mathcal{D}(w)$ of minimal type (see Definition \ref{type of diagrams}).

We prove Lemma \ref{bbjs} by an induction on the number of holes in $\Delta$. If $\Delta$ has no holes, then it will be a disk van Kampen diagram over \eqref{pre1} with boundary labeled by $uhv$ and thus $uhv$ represents $1\in K$.

Suppose that $\Delta$ has $k\geqslant 1$ holes. By Lemma \ref{cut}, there exists $\mu\in\Lambda$ and a connected component $c$ of $\partial_{int}\Delta$ such that $c$ is connected to an $H_{\mu}$-component of $\partial_{ext}\Delta$. Let $w$ be the label of $c$. Then $w$ is a word over $H_{\mu}$ representing an element $n\in N_{\mu}$. As $Lab(\partial_{ext}\Delta)\equiv uhv$, we can use Remark \ref{decompose loops on diagrams} to decompose $\partial_{ext}\Delta$ as the concatenation $p_up_hp_v$ of three paths $p_u,p_h$, and $p_v$ with $Lab(p_u)\equiv u, Lab(p_h)\equiv h,Lab(p_v)\equiv v$. Depending on where $c$ is connected to, there are three possible cases.


\textit{Case 1:} $c$ is connected to an $H_{\mu}$-component of $p_u$.

In other words, $u$ can be decomposed as $u\equiv u_1h_1u_2$ with $h_1\in H_{\mu}\setminus\{1\}$, and $p_u$ can be decomposed as a concatenation $p_{u_1}p_{h_1}p_{u_2}$ of three paths $p_{u_1},p_{h_1}$, and $p_{u_2}$ such that $Lab(p_{u_1})\equiv u_1, Lab(p_{h_1})\equiv h_1,Lab(p_{u_2})\equiv u_2$ and $c$ is connected to $p_{h_1}$ (see Remark \ref{decompose paths on diagrams}). By Lemma \ref{adding path}, passing to an equivalent diagram if necessary, we may assume that there exists a path $p_{h_2}$ in $\Delta$ with $Lab(p_{h_2})\equiv h_2\in H_{\mu}$, connecting the common vertex of $p_{h_1}$ and $p_{u_1}$ to a vertex of $c$. Note that the conjugate $n_1=h_2nh_2^{-1}\in N_{\mu}$. Let $h_3$ be the letter from $H_{\mu}$ such that $h_3=_Gn_1h_1$. Then
\begin{equation}\label{used in bbjs}
uhv\equiv u_1h_1u_2hv=_G (u_1n^{-1}_1u^{-1}_1)(u_1h_3u_2hv).
\end{equation}

As $h_1\neq 1$, we have $\|u_1\|\leqslant \|u\|-1$ and thus $\|u_1n^{-1}_1u^{-1}_1\|\leqslant 2\|u_1\|-1<2\|u\|+1$. Note that $u_1n^{-1}_1u^{-1}_1\in \ll\mathcal{N}\rr$. By the induction hypothesis, $u_1n^{-1}_1u^{-1}_1\in K$.

Let $u_4\equiv u_1h_3u_2$. Note that $\|u_4\|\leqslant \|u\|$. As $uhv,u_1n^{-1}_1u^{-1}_1\in \ll\mathcal{N}\rr$, it follows from \eqref{used in bbjs} that $u_4hv\in \ll\mathcal{N}\rr$. If $\|u_4\|<\|u\|$, then $\|u_4hv\|<2\|u\|+1$ and thus $u_4hv\in K$, by assumption. So let us assume that $\|u_4\|=\|u\|$. By Lemma \ref{u4}, $u_4\in T^{ex}_{\lambda}$. Let $\Sigma$ be a disc van Kampen diagram over \eqref{pre1} such that
$$Lab(\partial\Sigma)\equiv h_2wh^{-1}_2h_1h^{-1}_3.$$
Cut $\Delta$ along the path $p_{h_2}$ to produce a diagram $\Delta_1\in\mathcal{D}$ with 
$$Lab(\partial_{ext}\Delta_1)\equiv u_1h_2wh^{-1}_2h_1u_2hv.$$
Glue $\Sigma$ to $\Delta_1$ by identifying the paths with label $h_2wh^{-1}_2h_1$ (perform refinements if the non-essential edges of the two paths do not match) to construct a diagram $\Delta_2\in\mathcal{D}$ with $$Lab(\partial_{ext}\Delta_2)\equiv u_4hv$$
(see Figure \ref{cut1}). Note that the number of holes in $\Delta_2$ is strictly less than that of $\Delta$. By the induction hypothesis, $u_4hv\in K$. By \eqref{used in bbjs}, $uhv$ is a product of elements of $K$.

\textit{Case 2:} $c$ is connected to an $H_{\mu}$-component of $p_v$.

This case is symmetric to Case 1 and the proof is left to the reader.


\textit{Case 3}: $c$ is connected to $p_h$.

In other words, $\mu=\lambda$ and $h\in H_{\lambda}\setminus\{1\}$. By Lemma \ref{adding path} and passing to an equivalent diagram if necessary, we may assume that there exists a path in $\Delta$, labeled by a letter $h_1\in H_{\lambda}$, connecting the common vertex of $p_h$ and $p_u$ to a vertex of $c$. Note that the conjugate $n_1=h_1nh_1^{-1}\in N_{\lambda}$. Let $h_2$ be a letter from $H_{\lambda}$ such that $h_2=_G n_1h$. Consider the equality
\begin{equation}\label{used in bbjs 1}
uhv=_G (un^{-1}_1u^{-1})(uh_2v).
\end{equation}

As $u\in T^{ex}_{\lambda}$, Lemma \ref{inK} implies that $un^{-1}_1u^{-1}\in K$. An analysis similar to the one in Case 1 (with $uh_2v$ in place of $u_4hv$) shows that $uh_2v\in K$. By \eqref{used in bbjs 1}, $uhv$ is a product of elements of $K$.
\end{proof}

\begin{figure}
{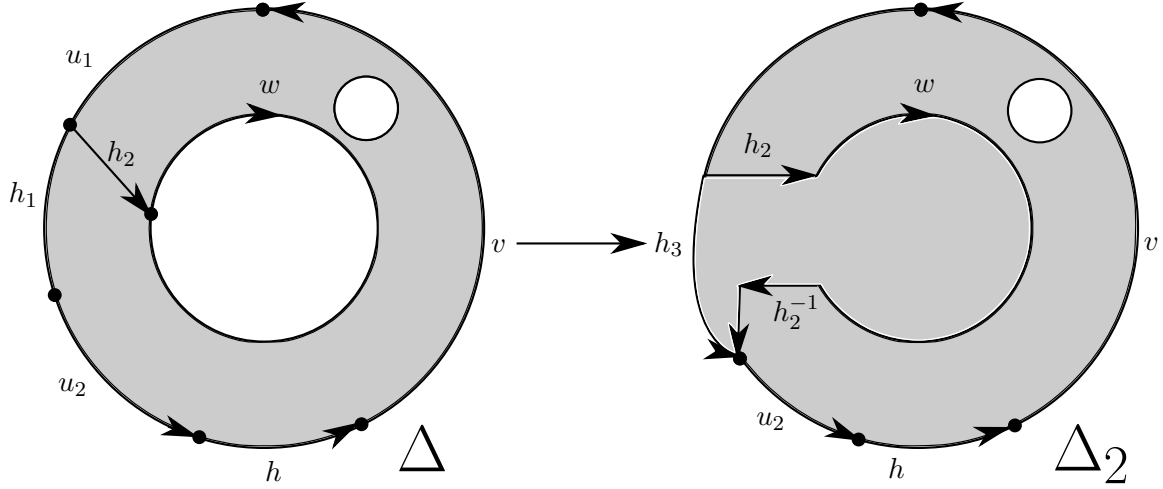}
\caption{An illustration of Case 1 in the proof of Lemma \ref{bbjs}}\label{cut1}
\end{figure}

\begin{definition}
Let $w$ be a word representing an element of $\ll\mathcal{N}\rr$. Define the number $k(w)$ to be the minimal number of holes of a diagram $\Delta\in\mathcal{D}(w)$. The \textit{type} of $w$ is the pair $\tau(w)=(\|w\|,k(w))$. We order the set of types lexicographically (see Definition \ref{type of diagrams}).
\end{definition}

\begin{remark}
If $w$ is a word representing an element of $\ll\mathcal{N}\rr$ and $\Delta$ is a diagram in $\mathcal{D}(w)$ of minimal type, then $\Delta$ necessarily has $k(w)$ holes.
\end{remark}

\begin{proposition}\label{generate}
$\ll\mathcal{N}\rr=K$.
\end{proposition}

\begin{proof}
Clearly, each of the groups $N^t_{\lambda},t\in T_{\lambda},\lambda\in\Lambda,$ is contained in $\ll\mathcal{N}\rr$ and thus $K \leqslant \ll\mathcal{N}\rr$. Let $w$ be a word over $X\sqcup \mathcal{H}$ such that $w\in\ll\mathcal{N}\rr$. Let us show that $w\in K$ by performing induction on the type of $w$. Note that the base case $\|w\|=k(w)=0$ is trivial.

Suppose that, for every word $w'$ over $X\sqcup\mathcal{H}$ with $w'\in \ll\mathcal{N}\rr$, $\tau(w')<\tau(w)$ implies that $w'\in K$. If $w$ is not a geodesic word, the induction hypothesis will imply $w\in K$. Thus, we may assume that $w$ is geodesic. Consider a diagram $\Delta\in\mathcal{D}(w)$ of minimal type.

By Lemma \ref{cut}, there exist $\lambda\in\Lambda$ and a connected component $c$ of $\partial_{int}\Delta$ connected to an $H_{\lambda}$-component of $\partial_{ext}\Delta$. In other words, $w$ can be decomposed as $uhv$ with $h\in H_{\lambda}\setminus\{1\}$ ($u,v$ are allowed to be empty words), and $\partial_{ext}\Delta$ can be decomposed as a concatenation $p_up_hp_v$ of three paths $p_u,p_h$, and $p_v$ such that $Lab(p_u)=u,Lab(p_h)=h,Lab(p_v)=v$ and $c$ is connected to $p_h$ (see Remark \ref{decompose loops on diagrams}). By Lemma \ref{adding path} and passing to an equivalent diagram if necessary, we may assume that there exists a path $p_{h_1}$ in $\Delta$ with $Lab(p_{h_1})\equiv h_1\in H_{\lambda}$, connecting the common vertex of $p_h$ and $p_u$ to a vertex of $c$.

Note that, as $h\neq 1$, at least one of $\|u\|$ and $\|v\|$ is at most $(\|w\|-1)/2$. Without loss of generality, we may assume that $\|v\|\leqslant (\|w\|-1)/2$. The case $\|u\|\leqslant (\|w\|-1)/2$ can be analyzed in almost the same way (or just by considering $w^{-1}$ and reversing every edge of $\Delta$ if one wishes).

Let $w_1\equiv Lab(c)$. Thus, $w_1\in N_{\lambda}$. Let $h_2$ be a letter from $H_{\lambda}$ such that $h_2=_G hh_1nh^{-1}_1$. There exists $t\in T_{\lambda}$ such that $t$ and $v^{-1}$ are in the same left $H_{\lambda}\ll \mathcal{N} \rr$-coset. In other words, there exists $h_3\in H_{\lambda}$ such that $th_3v\in \ll\mathcal{N}\rr$. Let $n_1$ be a letter in $N_{\lambda}$ such that $n_1=_G h_3h_1n^{-1}h^{-1}_1h^{-1}_3$.

Consider the equality
\begin{equation}\label{factor}
w\equiv uhv=_G (uh_2v)(v^{-1}h^{-1}_3t^{-1})(tn_1t^{-1})(th_3v).
\end{equation}

Note that $uh_2v\in \ll\mathcal{N}\rr$, as all other brackets in \eqref{factor} represents elements of $\ll\mathcal{N}\rr$. As in the proof of Lemma \ref{bbjs}, let $\Sigma$ be a disc van Kampen diagram over \eqref{pre1} with $$Lab(\partial\Sigma)\equiv hh_1w_1h^{-1}_1h^{-1}_2.$$
Cut $\Delta$ along $p_{h_1}$ to produce a diagram $\Delta_1\in\mathcal{D}$ with $Lab(\partial_{ext}\Delta_1)\equiv uhh_1w_1h^{-1}_1v$. Glue $\Delta_1$ to $\Sigma$, identifying the paths labeled by $hh_1w_1h^{-1}_1$ (perform refinements if the non-essential edges of the two paths do not match). Denote the resulting diagram by $\Delta_2$. Clearly, $\Delta_2\in\mathcal{D}$ and $Lab(\partial_{ext}\Delta_2)\equiv uh_2v$. Note that the number of holes in $\Delta_2$ is strictly less than that of $\Delta$, and that $\|uh_2v\|\leqslant \|u\|+\|v\|+1=\|uhv\|$, as $uhv$ is a geodesic word. Thus, $\tau(uh_2v)<\tau(w)$ and the induction hypothesis implies $uh_2v\in K$.

Clearly, $tn_1t^{-1}\in K$. Note also that $th_3v\in K$. Indeed, if either $\|t\|<\|v\|$ or $h_3=1$, then $\|th_3v\|<2\|v\|+1=\|w\|$ and the induction hypothesis implies that $th_3v\in K$. If $\|t\|=\|v\|$ and $h_3\neq 1$, then Lemma \ref{bbjs} implies $th_3v\in K$.

As $v^{-1}h^{-1}_3t^{-1}\equiv (th_3v)^{-1}$, we also have $v^{-1}h^{-1}_3t^{-1}\in K$. By \eqref{factor}, $w$ is a product of elements of $K$.
\end{proof}

The cutting process in the proof of Lemma \ref{generate} is exactly the same as the one for Lemma \ref{bbjs}. See Figure \ref{cut1} for an illustration.

The goal of the rest of this section is to prove the following.

\begin{proposition}\label{free}
$\ll\mathcal{N}\rr=\prod^{\ast}_{\lambda\in\Lambda,t\in T_{\lambda}}N^t_{\lambda}$.
\end{proposition}

\begin{proof}
Assume, for the contrary, that there exists a word 
\begin{equation}\label{star}
z\equiv \prod_{i=1}^k t_i n_i t_i^{-1}
\end{equation}
representing $1\in G$ such that
\begin{enumerate}
\item[(Z1)] $k\geqslant 2$;
\item[(Z2)] for $i=1,...,k$, there exists $\lambda_i\in\Lambda$ such that $n_i\in N_{\lambda_i}\setminus\{1\}$ and $t_i\in T_{\lambda_i}$;
\item[(Z3)] for $i=1,...,k$, either $\lambda_i\neq\lambda_{i+1}$ or $t_i\not\equiv t_{i+1}$ (subscripts are modulo $k$, i.e., $n_{k+1}=n_1$, $t_0=t_k$, etc.).
\end{enumerate}

Without loss of generality, we may also assume

\begin{enumerate}
\item[(Z4)] $z$ is minimal, i.e., has the minimal $k$ among all other words of the form \eqref{star} representing $1$ in $G$ and satisfying (Z1), (Z2), and (Z3).
\end{enumerate}

The main idea of the proof of Lemma \ref{free} is to show that the existence of such a word $z$ contradicts Lemma \ref{length}. For this purpose, it is convenient to first cyclically permute $z$ and consider the word
$$w\equiv t^{-1}_k(\prod_{i=1}^{k-1} t_i n_i t_i^{-1})t_kn_k.$$

In what follows, subscripts are modulo $k$. Let $p_w$ be the path in $\Gamma(G,X\sqcup \mathcal{H})$ with $Lab(p)\equiv w$ and $p^-=1$. We use $p_{n_i},p_{t^{\pm1}_i}$ to denote subpaths of $p_w$ labeled by $n_i,t^{\pm1}_i$, respectively. More precisely, $p_{n_i}$ (resp. $p_{t_i},p_{t^{-1}_i}$) will denote the path in the Cayley graph $\Gamma(G,X\sqcup\mathcal{H})$ with $Lab(p_{n_i})=n_i$ (resp. $Lab(p_{t_i})=t_i,Lab(p_{t^{-1}_i})=t^{-1}_i$) and $p^-_{n_i}=t^{-1}_k(\prod_{j=1}^{i-1}t_jn_jt^{-1}_j)t_i$ (resp. $p^-_{t_i}=t^{-1}_k(\prod_{j=1}^{i-1}t_jn_jt^{-1}_j),p^-_{t^{-1}_i}=t^{-1}_k(\prod_{j=1}^{i-1}t_jn_jt^{-1}_j)t_in_i$).

Recall that the collection $\{T_{\lambda}\}_{\lambda\in\Lambda}$ satisfies (P1), (P2), and (P3). Note that, for every $\lambda\in\Lambda$ and every word $t\in T_{\lambda}$, the word $t$ does not end with a letter from $H_{\lambda}$, by (P2). It follows that $p_{n_i}$ is an $H_{\lambda_i}$-component of $p_w$ for $i=1,...,k$. Being a cyclic permutation of $z$, the word $w$ represents $1$ in $G$ and thus the terminal vertex of $p_w$ is $1$. Hence, $p_w$ is a geodesic $3k$-gon. As $\widehat{\ell}_{\lambda_i}(p_{n_i})=\widehat{d}_{\lambda_i}(1,n_i)$ for $i=1,...,k$, by Lemma \ref{length} and (24D), there exists some $i\in \{1,...,k\}$ such that $p_{n_i}$ is not an isolated $H_{\lambda_i}$-component of $p_w$.

The rest of the proof is divided into several lemmas. All of them are stated under the assumptions (and using the notations) of Proposition \ref{free}.

\begin{lemma}\label{eliminate}
If $p_{n_i}$ is not an isolated $H_{\lambda_i}$-component of $p_w$ for some $i\in \{1,...,k\}$, then there are only three possibilities:
\begin{enumerate}
\item[(a)] $p_{n_i}$ is connected to an $H_{\lambda_i}$-component of $p_{t_{i+1}}$, but not connected to any $H_{\lambda_i}$-component of $p_{t^{-1}_{i-1}}$.
\item[(b)] $p_{n_i}$ is connected to an $H_{\lambda_i}$-component of $p_{t^{-1}_{i-1}}$, but not connected to any $H_{\lambda_i}$-component of $p_{t_{i+1}}$.
\item[(c)] $p_{n_i}$ is connected to both an $H_{\lambda_i}$-component of $p_{t_{i+1}}$ and an $H_{\lambda_i}$-component of $p_{t^{-1}_{i-1}}$.
\end{enumerate}
\end{lemma}

\begin{proof}
Without loss of generality, let us assume that $p_{n_1}$ is not isolated in $p_w$. There are six cases to consider (see Figure \ref{illustrateeliminate} for an illustration).

\textit{Case 1:} $p_{n_1}$ is connected to an $H_{\lambda_1}$-component of either $p_{t_1}$ or $p_{t_1^{-1}}$. In this case, some terminal segment of $t_1$ represents an element of $H_{\lambda_1}$, which contradicts (P2).

\textit{Case 2:} $p_{n_1}$ is connected to either $p_{n_2}$ or $p_{n_k}$. If $p_{n_1}$ is connected to $p_{n_2}$, then $\lambda_1=\lambda_2$, which in turn implies $t_1,t_2\in T_{\lambda_1}$. The assumption that $p_{n_1}$ is connected to $p_{n_2}$ also implies $t^{-1}_1t_2\in H_{\lambda_1}$. By (P1), $t_1\equiv t_2$, which contradicts (Z3) as $\lambda_1=\lambda_2$. The analysis for the subcase where $p_{n_1}$ is connected to $p_{n_k}$ is similar.

\textit{Case 3:} $p_{n_1}$ is connected to $p_{n_i}$ for some $i\in \{3,...,k-1\}$. In other words, there exists $h\in H_{\lambda_1}$ such that the word
\begin{equation*}
u\equiv t_1^{-1}(\prod_{j=2}^{i-1}t_jn_jt_j^{-1})t_ih
\end{equation*}
represents $1$ in $G$. As $\prod_{j=2}^{i-1}t_jn_jt_j^{-1}\in \ll\mathcal{N}\rr\lhd G$, we have $t_1^{-1}t_i\in H_{\lambda_1}\ll\mathcal{N}\rr$. The assumption that $p_{n_1}$ is connected to $p_{n_i}$ also implies $n_1,n_i\in N_{\lambda_1}$ and thus $t_1,t_i\in T_{\lambda_1}$. By (P1), $t_1\equiv t_i$. Thus, the word
\begin{equation*}
u'\equiv t_1ht_1^{-1}(\prod_{j=2}^{i-1}t_jn_jt_j^{-1})
\end{equation*}
is a cyclic permutation of $u$ and represents $1$ in $G$. It follows that $t_1ht_1^{-1}\in\ll\mathcal{N}\rr$. By Theorem \ref{Dehn filling}, Remark \ref{constant}, and Condition (24D), we have $h\in N_{\lambda_1}$. Then the word $t_1ht_1^{-1}(\prod_{j=2}^{i-1}t_jn_jt_j^{-1})$ represents $1$ in $G$, contradicting (Z4).

\textit{Case 4:} $p_{n_1}$ is connected to an $H_{\lambda_1}$-component of $p_{t_i}$ for some $i\in \{3,...,k\}$. Thus, $t_i$ can be decomposed as $t_i\equiv t'_ih't''_i$ with $h'\in H_{\lambda_1}\setminus\{1\}$ and there exists $h\in H_{\lambda_1}$ such that the word
\begin{equation*}
u\equiv t_1^{-1}(\prod_{j=2}^{i-1}t_jn_jt_j^{-1})t'_ih
\end{equation*}
represents $1$ in $G$. By (P3), $t'_i$ belongs to $T_{\lambda_1}$. Arguing as in Case 3, we conclude that the word $t_1ht_1^{-1}(\prod_{j=2}^{i-1}t_jn_jt_j^{-1})$ represents $1$ in $G$, contradicting (Z4). 


\textit{Case 5:} $p_{n_1}$ is connected to an $H_{\lambda_1}$-component of $p_{t_i^{-1}}$ for some $i\in \{2,...,k-1\}$. This case can be reduced to Case 4 by considering $w^{-1}$.


Thus, the only possibilities left are (a), (b), and (c).
\end{proof}

\begin{figure}
\begin{center}
\resizebox{0.95\linewidth}{!}{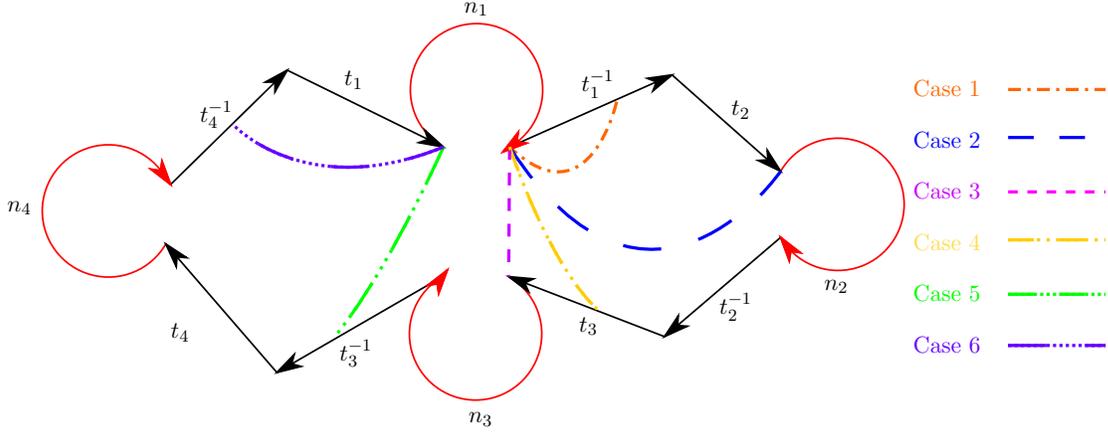}
\end{center}
\caption{Cases 1 through 6 in the proof of Lemma \ref{eliminate}}\label{illustrateeliminate}
\end{figure}

\begin{lemma}\label{prefix 1}
If $p_{n_i}$ is connected to an $H_{\lambda_i}$-component of $p_{t_{i+1}}$, then $t_{i+1}$ can be decomposed as $t_{i+1}\equiv uhv$ with $h\in H_{\lambda_i}\setminus\{1\}$ ($u,v$ are allowed to be empty words), $t_i\equiv u$, and $\widehat{d}_{\lambda_i}(1,n_ih)>12D$.
\end{lemma}

\begin{proof}
By Definition \ref{connect}, $t_{i+1}$ can be decomposed as $t_{i+1}\equiv uhv$ with $h\in H_{\lambda_i}\setminus\{1\}$ such that $p_{n_i}$ is connected to the path $p_h$ in $\Gamma(G,X\sqcup\mathcal{H})$ with $Lab(p_h)\equiv h$ and $p^-_h=t^{-1}_k(\prod_{j=1}^i t_jn_jt^{-1}_j)u$. By (P3), $u\in T_{\lambda_i}$. The assumption that $p_{n_i}$ is connected to $p_h$ also implies $t^{-1}_iu\in H_{\lambda_i}$ and thus $t_i\equiv u$, by (P1). Another consequence of (P3) is $$\widehat{d}_{\lambda_i}(1,h)\leqslant\widehat{d}_{\lambda_i}(1,h(h^{-1}n_ih))=\widehat{d}_{\lambda_i}(1,n_ih).$$
Therefore, the triangle inequality implies $$\widehat{d}_{\lambda_i}(1,n_i)\leqslant \widehat{d}_{\lambda_i}(1,n_ih)+\widehat{d}_{\lambda_i}(1,h^{-1})=\widehat{d}_{\lambda_i}(1,n_ih)+\widehat{d}_{\lambda_i}(1,h)\leqslant 2\widehat{d}_{\lambda_i}(1,n_ih)$$
and thus 
$$\widehat{d}_{\lambda_i}(1,n_ih)\geqslant\widehat{d}_{\lambda_i}(1,n_i)/2>12D,$$
by (24D).
\end{proof}

The next lemma follows from Lemma \ref{prefix 1} by considering $w^{-1}$.

\begin{lemma}\label{prefix 2}
If $p_{n_i}$ is connected to an $H_{\lambda_i}$-component of $p_{t^{-1}_{i-1}}$, then $t_{i-1}$ can be decomposed as $t_{i-1}\equiv uhv$ with $h\in H_{\lambda_i}\setminus\{1\}$ ($u,v$ are allowed to be empty words), $t_i\equiv u$, and $\widehat{d}_{\lambda_i}(1,h^{-1}n_i)>12D$.
\end{lemma}


\begin{lemma}\label{disconnect}
If $p_{n_i}$ is connected to an $H_{\lambda_i}$-component of $p_{t_{i+1}}$, then $p_{n_{i+1}}$ is not connected to any $H_{\lambda_{i+1}}$-component of $p_{t^{-1}_i}$. If $p_{n_i}$ is connected to an $H_{\lambda_i}$-component of $p_{t^{-1}_{i-1}}$, then $p_{n_{i-1}}$ is not connected to any $H_{\lambda_{i-1}}$-component of $p_{t_i}$.
\end{lemma}

\begin{proof}
If $p_{n_i}$ is connected to an $H_{\lambda_i}$-component of $p_{t_{i+1}}$, then $\|t_i\|<\|t_{i+1}\|$ by Lemma \ref{prefix 1}. If, in addition, $p_{n_{i+1}}$ is connected to an $H_{\lambda_{i+1}}$-component of $p_{t^{-1}_i}$, then $\|t_{i+1}\|<\|t_i\|$ by Lemma \ref{prefix 2}, a contradiction.

The second assertion of the Lemma can be proved by considering $w^{-1}$.
\end{proof}

Recall that we assume the existence of a word $z$ satisfying (Z1) through (Z4) and construct $w,p_w$ from $z$. The previous several lemmas reveal some properties of $p_w$ and we are now ready to construct a geodesic polygon $p$ from $p_w$ so that $p$ violates Lemma \ref{length}, and then we can conclude that $z$ does not exist and prove Proposition \ref{free}. The idea is to merge all $H_{\lambda_i}$-components connected to $p_{n_i}$ to form an isolated $H_{\lambda_i}$-component for $i=1,...,k-1$. Of course, one can also merge $p_{n_k}$ with the $H_{\lambda_k}$-components connected to it. We do not perform this merging only because it makes the construction more complicated. Pick elements $h_1,...,h_{k-1}\in \mathcal{H}$ and $g_{1,1},g_{1,2},g_{2,1},g_{2,2},...,g_{k-1,1},g_{k-1,2}\in G$ by the following procedure.
\begin{procedure}\label{pro.}
For $i=1,...,k-1$, perform the following.
\begin{enumerate}
\item[(a)] If $p_{n_i}$ is an isolated $H_{\lambda_i}$-component in $p_w$, let $g_{i,1}\in G$ (resp. $g_{i,2}\in G$) be represented by the word $t^{-1}_k(\prod_{j=1}^{i-1} t_j n_j t_j^{-1})t_i$ (resp. $t^{-1}_k(\prod_{j=1}^{i-1} t_j n_j t_j^{-1})t_in_i$), and let $h_i=n_i$.
\item[(b)] If, in $p_w$, $p_{n_i}$ is connected to an $H_{\lambda_i}$-component of $p_{t_{i+1}}$, but not connected to any $H_{\lambda_i}$-component of $p_{t^{-1}_{i-1}}$, then by Lemma \ref{prefix 1}, $t_{i+1}$ can be decomposed as $t_{i+1}\equiv u_ih'_iv_i$ with $h'_i\in H_{\lambda_i}\setminus\{1\}$, $t_i\equiv u_i$, and $\widehat{d}_{\lambda_i}(1,n_ih'_i)>6D$. Let $h_i$ be a letter from $H_{\lambda_i}$ such that $h_i=_G n_ih'_i$, and let $g_{i,1}\in G$ (resp. $g_{i,2}\in G$) be represented by the word $t^{-1}_k(\prod_{j=1}^{i-1} t_j n_j t_j^{-1})t_i$ (resp. $t^{-1}_k(\prod_{j=1}^{i-1} t_j n_j t_j^{-1})t_ih_i$).
\item[(c)] If in $p_w$, $p_{n_i}$ is connected to an $H_{\lambda_i}$-component of $p_{t^{-1}_{i-1}}$, but not connected to any $H_{\lambda_i}$-component of $t_{i+1}$, then by Lemma \ref{prefix 2}, $t_{i-1}$ can be decomposed as $t_{i-1}\equiv u_ih'_iv_i$ with $h'_i\in H_{\lambda_i}\setminus\{1\}$, $t_i\equiv u_i$, and $\widehat{d}_{\lambda_i}(1,h'^{-1}_in_i)>6D$. Let $h_i$ be a letter from $H_{\lambda_i}$ such that $h_i=_G h'^{-1}_in_i$, and let $g_{i,1}\in G$ (resp. $g_{i,2}\in G$) be represented by the word $t^{-1}_k(\prod_{j=1}^{i-2} t_j n_j t_j^{-1})t_{i-1}n_{i-1}v^{-1}_i$ (resp. $t^{-1}_k(\prod_{j=1}^{i-2} t_j n_j t_j^{-1})t_{i-1}n_{i-1}v^{-1}_ih_i$).
\item[(d)] If in $p_w$, $p_{n_i}$ is connected to both an $H_{\lambda_i}$-component of $p_{t_{i+1}}$ and an $H_{\lambda_i}$-component of $p_{t^{-1}_{i-1}}$, then by Lemmas \ref{prefix 1} and \ref{prefix 2}, $t_{i+1}$ (resp. $t_{i-1}$) can be decomposed as $t_{i+1}\equiv u_ih'_iv_i$ (resp. $t_{i-1}\equiv u'_ih''_iv'_i$) with $h'_i\in H_{\lambda_i}\setminus\{1\}$ (resp. $h''_i\in H_{\lambda_i}\setminus\{1\}$), $t_i\equiv u_i$ (resp. $t_i\equiv u'_i$). Let $h_i$ be a letter from $H_{\lambda_i}$ such that $h_i=_G h''^{-1}_in_ih'_i$, and let $g_{i,1}\in G$ (resp. $g_{i,2}\in G$) be represented by the word $t^{-1}_k(\prod_{j=1}^{i-2} t_j n_j t_j^{-1})t_{i-1}n_{i-1}(v'_i)^{-1}$ (resp. $t^{-1}_k(\prod_{j=1}^{i-2} t_j n_j t_j^{-1})t_{i-1}n_{i-1}(v'_i)^{-1}h_i$).
\end{enumerate}
\end{procedure}

\begin{lemma}\label{order}
$g_{i,1}$ and $g_{i,2}$ are vertices on $p_w$ for $i=1,...,k-1$. Moreover, the order in which $p_w$ visits these vertices is $g_{1,1},g_{1,2},g_{2,1},g_{2,2},...,g_{k-1,1},g_{k-1,2}$.
\end{lemma}

\begin{proof}
The first assertion follows directly from the choices of those vertices. Clearly, the path $p_w$ visits $g_{i,1}$ before visiting $g_{i,2}$ for $i=1,...,k-1$. Thus, the second assertion will be proved once we show that, for all $i,j\in \{1,...,k-1\}$ with $i<j$, the path $p_w$ visits $g_{i,2}$ before visiting $g_{j,1}$.

Suppose, for the contrary, that for some $i,j\in \{1,...,k-1\}$ with $i<j$, the path $p_w$ visits $g_{j,1}$ before visiting $g_{i,2}$. By Lemma \ref{eliminate}, there is only one possibility for this case: $j=i+1$, $p_{n_i}$ is connected to an $H_{\lambda_i}$-component of $p_{t_{i+1}}$, and $p_{n_{i+1}}$ is connected to an $H_{\lambda_{i+1}}$-component of $p_{t^{-1}_i}$. By Lemma \ref{disconnect}, if  $p_{n_i}$ is connected to an $H_{\lambda_i}$-component of $p_{t_{i+1}}$, then $p_{n_{i+1}}$ is not connected to any $H_{\lambda_{i+1}}$-component of $p_{t^{-1}_i}$, a contradiction.
\end{proof}

\begin{lemma}\label{obvious}
For $i=1,...,k-2$, the subpath of $p_w$ from $g_{i,2}$ to $g_{i+1,1}$ consists of at most two geodesic segments.
\end{lemma}

Lemma \ref{obvious} follows immediately from the choices of the vertices $g_{i,1}$ and $g_{i,2}$, $1\leqslant i\leqslant k-1$. We are now ready to construct a geodesic polygon $p$ from $p_w$.

\begin{construction}\label{con.}
For $i=1,...,k-1$, let $p_{h_i}$ the edge of $\Gamma(G,X\sqcup \mathcal{H})$ with $Lab(p_{h_i})=h_i$ and $p^-_{h_i}=g_{i,1}$. Let $p$ be the path in $\Gamma(G,X\sqcup \mathcal{H})$ satisfying: $p^-$ is the identity vertex. $p$ first follows the path of $p_w$ (in the direction of $p_w$) until $p$ visits $g_{1,1}$, and then $p$ travels along $p_{h_1}$ and arrives at $g_{1,2}$. And then $p$ follows the path $p_w$ (in the direction of $p_w$) until $p$ arrives at $g_{2,1}$ (Lemma \ref{order} guarantees that $p$ will arrive at $g_{2,1}$), where $p$ travels along $p_{h_2}$ and then arrives at $g_{2,2}$. The path $p$ continues traveling in this manner until arriving at $g_{k-1,2}$. Finally, $p$ follows the path $p_w$ (in the direction of $p_w$) and comes back to the identity vertex.
\end{construction}

\begin{figure}
{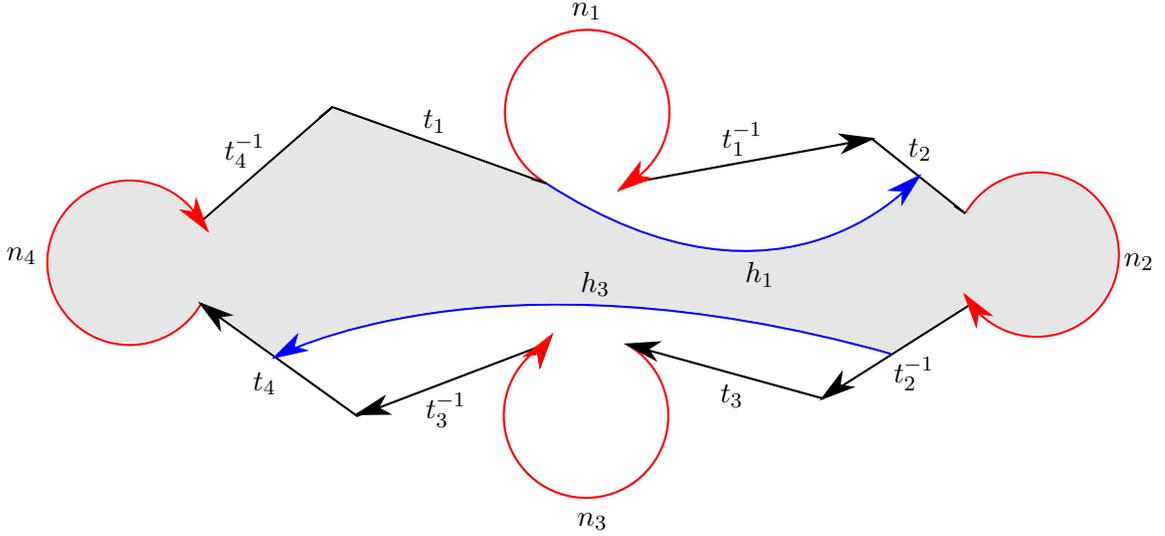}
\caption{The construction of $p$}\label{constructp}
\end{figure}

Figure \ref{constructp} illustrates how to construct the geodesic polygon $p$. In Figure \ref{constructp}, the outside boundary with label $t^{-1}_4t_1n_1t^{-1}_1t_2n_2t^{-1}_2t_3n_3t^{-1}_3t_4n_4$ is the geodesic polygon $p_w$. In the outside boundary, $p_{n_2}$ is an isolated $H_{\lambda_2}$-component, $p_{n_1}$ (resp. $p_{n_4}$) is connected to an $H_{\lambda_1}$-component (resp. $H_{\lambda_4}$-component) of $p_{t_2}$ (resp. $p_{t_1}$), and $p_{n_3}$ is connected to both an $H_{\lambda_3}$-component of $p_{t^{-1}_2}$ and an $H_{\lambda_3}$-component of $p_{t_4}$. By Lemma \ref{prefix 1}, $t^{-1}_1$ cancels with a prefix of $t_2$. After this cancellation, $p_{n_1}$ merges with an $H_{\lambda_1}$-component of $p_{t_2}$ to form $p_{h_1}$. Similarly, $p_{n_3}$ merges with both an $H_{\lambda_3}$-component of $p_{t^{-1}_2}$ and an $H_{\lambda_3}$-component of $p_{t_4}$ to form $p_{h_3}$. The merging process does nothing to $n_4$, although $n_4$ is not an isolated $H_{\lambda_4}$-component. Finally, $p_w$ becomes $p$, the boundary of the shaded region.

\begin{remark}
It follows directly from the above construction that $p_{h_i}$ is an isolated $H_{\lambda_i}$-component of $p$ for $i=1,...,k-1$.
\end{remark}

Note that the subpath of $p_w$ from $1$ to $g_{i,1}$ consists of at most $2$ geodesic segments, and the subpath of $p_w$ from $g_{k-1,2}$ to $1$ consists of at most $3$ geodesic segments. Together with Lemma \ref{obvious}, these observations imply that $p$ is a polygon in $\Gamma(G,X\sqcup \mathcal{H})$ with at most $3k$ geodesic sides.

Consider the following partition of $\{1,...,k-1\}=I_1\sqcup I_2$. A number $1\leqslant i\leqslant k-1$ belongs to $I_1$ if in $p_w$, $p_{n_i}$ is connected to both an $H_{\lambda_i}$-component of $p_{t^{-1}_{i-1}}$ and an $H_{\lambda_i}$-component of $p_{t_{i+1}}$. Otherwise, $i$ belongs to $I_2$.

\begin{lemma}
$card(I_1)\leqslant (k-1)/2$.
\end{lemma}

\begin{proof}
First suppose $card(I_1)>k/2$. Then there exists a number $i$ such that both $i$ and $i+1$ belong to $I_1$, contradicting Lemma \ref{disconnect}. Thus, $card(I_1)\leqslant k/2$.

Suppose $card(I_1)=k/2$. Then $k$ is even and $I_1=\{1,3,...,k-3,k-1\}$. For every even number $i\in\{2,4,...,k-2,k\}$, Lemma \ref{disconnect} implies that $p_{n_i}$ is an isolated $H_{\lambda_i}$-component of $p_w$. Note that $\widehat{\ell}_{\lambda_i}(p_{n_i})=\widehat{d}_{\lambda_i}(1,n_i)>24D$ for $i=1,...,k$, by (24D). Therefore, Lemma \ref{length}, applied to the geodesic $3k$-gon $p_w$, yields
$$\dfrac{24Dk}{2}<\widehat{\ell}_{\lambda_2}(p_{n_2})+\widehat{\ell}_{\lambda_4}(p_{n_4})+\cdot\cdot\cdot+\widehat{\ell}_{\lambda_{k-2}}(p_{n_{k-2}})+\widehat{\ell}_{\lambda_k}(p_{n_k})<3kD,$$
a contradiction.
\end{proof}

Thus, $card(I_2)=k-1-card(I_1)\geqslant (k-1)/2$. For each $i\in I_2$, $p_{h_i}$ is an isolated $H_{\lambda_i}$-component of $p$ with $\widehat{\ell}_{\lambda_i}(p_{h_i})=\widehat{d}_{\lambda_i}(1,h_i)>6D$, by Procedure \ref{pro.} and Construction \ref{con.}. Lemma \ref{length}, applied to the geodesic polygon $p$, yields
\begin{equation}\label{estimate}
6D(k-1)=12D(k-1)/2<\sum_{i\in I_2}\widehat{\ell}_{\lambda_i}(p_{h_i})\leqslant 3kD.
\end{equation}
In other words, $k<2$, contradicting (Z1). Proposition \ref{free} is proved.
\end{proof}

Finally, Theorem \ref{main} follows from Proposition \ref{generate} and Proposition \ref{free}.

\begin{remark}
The proof of Theorem \ref{main} implies that if $\{H_{\lambda}\}_{\lambda\in\Lambda}\hookrightarrow_{wh} (G,X)$ for some $X\subset G$, $N_{\lambda}\lhd H_{\lambda}$ for $\lambda\in\Lambda$, and (24D) holds, then for every collection $\{T_{\lambda}\}_{\lambda\in\Lambda}$ satisfying (P1), (P2), and (P3), we have
$$\ll\mathcal{N}\rr=\prod^{\ast}_{\lambda\in\Lambda,t\in T_{\lambda}}N^t.$$
\end{remark}

\begin{remark}
In fact, one can show that if $\{H_{\lambda}\}_{\lambda\in\Lambda}\hookrightarrow_{wh} (G,X)$ for some $X\subset G$, $N_{\lambda}\lhd H_{\lambda}$ for $\lambda\in\Lambda$, and following condition
\begin{enumerate}
    \item[(4D)] $\widehat{d}_{\lambda}(1,n)>4D$ for all $n\in N_{\lambda}\setminus\{1\}$ and $\lambda\in\Lambda$
\end{enumerate}
holds, then the triple $(G,\{H_{\lambda}\}_{\lambda\in\Lambda},\{N_{\lambda}\}_{\lambda\in\Lambda})$ possesses the Cohen-Lyndon property. For the proof, one needs to merge $p_{n_k}$ with the $H_{\lambda_k}$-components connected to it in the construction of $p$, and sharpen the coarse estimate \eqref{estimate}.
\end{remark}

\section{Relative relation modules}\label{sec.rm}
Let $H$ be a group with a normal subgroup $N$ and $R=H/N$. The \textit{relative relation module} $Rel(H,N)$ \textit{of the exact sequence} $$1\rightarrow N\rightarrow H\rightarrow R\rightarrow 1$$
is the abelianization $\overline{N}=N/[N,N]$ equipped with the $R$-action by conjugation. More precisely, denote by $\overline{n}$ the image of an element $n\in N$ under the quotient map $N\rightarrow \overline{N}$. Then there is an action of $H$ on $\overline{N}$ given by $h\circ \overline{n}=\overline{hnh^{-1}}$ for all $h\in H,\overline{n}\in \overline{N}$. Notice that if $h$ belongs to $N$, then $h\circ \overline{n}=\overline{hnh^{-1}}=\overline{h}\overline{n}\overline{h}^{-1}=\overline{n}$ for all $\overline{n}\in \overline{N}$, as $\overline{h}$ commutes with $\overline{n}$. Hence, the action of $H$ gives rises to an action of $R$, turning $\overline{N}$ into a $\mathbb{Z}R$-module. If $H$ is a free group, then $Rel(H,N)$ is called a \textit{relation module}.

The main goal of this section is to prove Proposition \ref{genmodule}, which, together with Theorem \ref{simmain}, implies Corollary \ref{module}.

\begin{proposition}\label{genmodule}
Suppose that $G$ is a group, that $\{H_{\lambda}\}_{\lambda\in\Lambda}$ is a family of subgroups of $G$, and that $N_{\lambda}$ is a normal subgroup of $H_{\lambda}$ for every $\lambda\in\Lambda$. Let $\mathcal{N}=\bigcup_{\lambda\in\Lambda}N_{\lambda}$, $Q=G/\ll \mathcal{N}\rr$, and $R_{\lambda}=H_{\lambda}/N_{\lambda}$ for $\lambda\in\Lambda$. If $N_{\lambda}\neq\{1\}$ for every $\lambda\in\Lambda$ and the triple $(G,\{H_{\lambda}\}_{\lambda\in\Lambda},\{N_{\lambda}\}_{\lambda\in\Lambda})$ has the Cohen-Lyndon property, then 
\begin{enumerate}
\item[(a)] for every $\lambda\in\Lambda$, the natural map $R_{\lambda}\rightarrow Q$ is injective (i.e., $H_{\lambda}\cap \ll \mathcal{N} \rr = N_{\lambda}$), identifying $R_{\lambda}$ with a subgroup of $Q$;
\item[(b)] $Rel(G,\ll \mathcal{N} \rr)\cong\bigoplus_{\lambda\in\Lambda}Ind_{R_{\lambda}}^Q Rel(H_{\lambda},N_{\lambda})$ as $\mathbb{Z}Q$-modules.
\end{enumerate}
\end{proposition}

\begin{remark}
If $N_{\lambda_0}=\{1\}$ for some $\lambda_0\in\Lambda$, then we can consider the subset $\Lambda'$ such that $N_{\lambda}\neq \{1\}$ for every $\lambda\in\Lambda'$. Proposition \ref{genmodule} can then be applied to $(G,\{H_{\lambda}\}_{\lambda\in\Lambda'},\{N_{\lambda}\}_{\lambda\in\Lambda'})$.
\end{remark}


Suppose that the assumptions of Proposition \ref{genmodule} are satisfied. Let $T_{\lambda},\lambda\in\Lambda$, be the transversals provided by Definition \ref{CLp}. Fix some $\lambda\in\Lambda$ for the moment. Suppose $h\in H_{\lambda}\cap \ll\mathcal{N}\rr$. Then $h\in N_{\ll\mathcal{N}\rr}(N_{\lambda})$, the normalizer of $N_{\lambda}$ in $\ll\mathcal{N}\rr$. Note that
$$\ll\mathcal{N}\rr=\prod^{\ast}_{\mu\in\Lambda,t\in T_{\mu}}N^t_{\mu}=N_{\lambda}\ast (\prod^{\ast}_{t\in T_{\lambda}\setminus\{1\}}N^t_{\lambda}\ast \prod^{\ast}_{\mu\in\Lambda\setminus\{\lambda\},t\in T_{\mu}}N^t_{\mu})$$
and $N_{\lambda}\neq\{1\}$. Note also the following general fact.

\begin{lemma}
Let $A,B\neq\{1\}$ be groups. Then $N_{A\ast B}(A)=A$.
\end{lemma}

\begin{proof}
Suppose that there exists $a\in A\setminus\{1\}$ and $g\in A\ast B\setminus A$ such that $a^g\in A$. Consider the Bass-Serre tree $Tr$ corresponding to $A\ast B$. The vertex group $A$ fixes a vertex $v$ of $Tr$ and thus $a^g$ fixes $v$. Clearly, the vertex $g\circ v$ is also fixed by $a^g$. As $g\in A\ast B\setminus A$, $g\circ v\neq v$ and thus $a^g$ fixes a non-trivial path between $v$ and $g\circ v$. In particular, $a^g$ fixes an edge of $Tr$ and thus conjugates into the unique edge subgroup $\{1\}$ of $A\ast B$. It follows that $a^g=1$, which is in contradiction with $a\neq 1$.
\end{proof}

Therefore, $N_{\ll\mathcal{N}\rr}(N_{\lambda})=N_{\lambda}$ and $h\in N_{\lambda}$. We conclude:

\begin{lemma}
For every $\lambda\in\Lambda$, $H_{\lambda}\cap \ll \mathcal{N} \rr = N_{\lambda}$.
\end{lemma}





Let us consider the relative relation modules $Rel(G,\ll\mathcal{N}\rr)$ and $Rel(H_{\lambda},N_{\lambda}),\lambda\in\Lambda$. For every $\lambda\in\Lambda$, let $M_{\lambda}$ be the subgroup of $G$ generated by $N^t_{\lambda},t\in T_{\lambda}$. Note that $M_{\lambda}=\prod^{\ast}_{t\in T_{\lambda}}N^t_{\lambda}$ for every $\lambda\in\Lambda$, as $\ll\mathcal{N}\rr=\prod^{\ast}_{\lambda\in\Lambda,t\in T_{\lambda}}N^t_{\lambda}$. Note also that $\ll\mathcal{N}\rr=\prod^{\ast}_{\lambda\in\Lambda}M_{\lambda}$. 

For every $\lambda\in\Lambda$, the composition of natural maps $M_{\lambda}\hookrightarrow \ll\mathcal{N}\rr \rightarrow \overline{\ll\mathcal{N}\rr}$ maps $M_{\lambda}$ into the abelian group $\overline{\ll\mathcal{N}\rr}$ and thus factors through
$$i_{\lambda}:\overline{M_{\lambda}}\rightarrow \overline{\ll\mathcal{N}\rr}.$$
The homomorphisms $i_{\lambda},\lambda\in\Lambda,$ extend to an abelian group homomorphism
$$i:\bigoplus_{\lambda\in\Lambda}\overline{M_{\lambda}}\rightarrow \overline{\ll\mathcal{N}\rr}.$$
It is well-known that $i$ is an abelian group isomorphism (for example, see \cite[Problem 4 of Exercise 6.2]{robinson2012course}). Thus, we identify $\overline{M_{\lambda}}$ with its image $i_{\lambda}(\overline{M_{\lambda}})$ for every $\lambda\in\Lambda$ and write
$$Rel(G,\ll\mathcal{N}\rr)=\overline{\ll\mathcal{N}\rr}=\bigoplus_{\lambda\in\Lambda}\overline{M_{\lambda}}.$$

Fix $\lambda\in\Lambda$ for the moment. By the same argument as the one above, we write
$$\overline{M_{\lambda}}=\bigoplus_{t\in T_{\lambda}}\overline{N^t_{\lambda}}.$$

\begin{lemma}\label{permute factors}
$\overline{M_{\lambda}}$ is a $\mathbb{Z}Q$-submodule of $Rel(G,\ll\mathcal{N}\rr)=\bigoplus_{\lambda\in\Lambda}\overline{M_{\lambda}}$. The $Q$-action on $\overline{M_{\lambda}}$ transitively permutes the summands $\overline{N^t_{\lambda}},t\in T_{\lambda},$ and its isotropy group of $\overline{N_{\lambda}}$ is $R_{\lambda}$, i.e., an element $q\in Q$ satisfies $q\circ \overline{n}\in \overline{N_{\lambda}}$ for all $\overline{n}\in \overline{N_{\lambda}}$ if and only if $q\in R_{\lambda}$.
\end{lemma}

\begin{proof}
Fix $t_0\in T_{\lambda}$ and $g\in G$. There exists $t_1\in T_{\lambda},h\in H_{\lambda},$ and $m\in \ll\mathcal{N} \rr$ such that
\begin{equation}\label{find s}
gt_0=t_1hm.
\end{equation}

Consider the summand $\overline{N^{t_0}_{\lambda}}$. For all $n\in N_{\lambda}$,
$$g\circ\overline{t_0nt^{-1}_0}=\overline{gt_0nt^{-1}_0g^{-1}}=\overline{t_1hmnm^{-1}h^{-1}t_1^{-1}}=\overline{t_1hnh^{-1}t_1^{-1}}\in \overline{N^{t_1}_{\lambda}},$$
where the fact that the action of $\ll\mathcal{N}\rr$ acts trivially on $Rel(G,\ll N \rr)$ is used in the second equality. Hence, $g\circ (\overline{N^{t_0}_{\lambda}})\subset \overline{N^{t_1}_{\lambda}}$. As $\overline{M_{\lambda}}=\bigoplus_{t\in T_{\lambda}}\overline{N^t_{\lambda}}$, it follows that $\overline{M_{\lambda}}$ is $G$-invariant and thus $\overline{M_{\lambda}}$ is also $Q$-invariant.

The above paragraph shows that $g$ maps $\overline{N^{t_0}_{\lambda}}$ into $\overline{N^{t_1}_{\lambda}}$. Actually, $g(\overline{N^{t_0}_{\lambda}})=\overline{N^{t_1}_{\lambda}}$. Indeed, given $n\in N_{\lambda}$, we find an element $x$ of $N^{t_0}_{\lambda}$ such that $g\circ \overline{x}=\overline{n^{t_1}}$. Let $x=n^{t_0h^{-1}}$. Note that $n^{h^{-1}}\in N_{\lambda}$, as $N_{\lambda}$ is normal in $H_{\lambda}$. Thus, $x\in N^{t_0}_{\lambda}$. Direct computation shows
$$g\circ\overline{x}=\overline{gxg^{-1}}=\overline{gt_0(h^{-1}nh)t^{-1}_0g^{-1}}=\overline{t_1hm(h^{-1}nh)m^{-1}h^{-1}t^{-1}_1}=\overline{t_1h(h^{-1}nh)h^{-1}t^{-1}_1}=\overline{n^{t_1}},$$
where the fact that the action of $\ll\mathcal{N}\rr$ on $Rel(G,\ll N \rr)$ is trivial is used in the second equality. Hence, $g\circ \overline{x}=\overline{n^{t_1}}$.

As a consequence, $g\circ (\overline{N^{t_0}_{\lambda}})=\overline{N^{t_1}_{\lambda}}$, i.e., the action of $G$ on $\overline{M_{\lambda}}$ permutes the summands $\overline{N^t_{\lambda}},t\in T_{\lambda}$. In fact, this permutation is transitive: Let $t$ be any element of $T_{\lambda}$. We wish to find an element of $G$ which maps $\overline{N^{t_0}_{\lambda}}$ to $\overline{N^t_{\lambda}}$. This can be done by $tt^{-1}_0$:
$$tt^{-1}_0\circ \overline{N^{t_0}_{\lambda}}=\overline{N^t_{\lambda}}.$$

Thus, the action of $G$ on $\overline{M_{\lambda}}$ transitively permutes the summands $\overline{N^t_{\lambda}},t\in T_{\lambda}$. The same is thus true for the action of $Q$ on $\overline{M_{\lambda}}$.

Clearly, for the action of $G$ on $\overline{M_{\lambda}}$, the isotropy group of $\overline{N_{\lambda}}$ contains $H_{\lambda}\ll \mathcal{N} \rr$. Observe that in equation \eqref{find s}, if $t_0=1$ and $g\not\in H_{\lambda}\ll \mathcal{N} \rr$, then $t_1\neq 1$ as $t^{-1}_1g\in H_{\lambda}\ll \mathcal{N} \rr$. It follows that
$$g\circ \overline{N_{\lambda}}=\overline{N^{t_1}_{\lambda}}\neq \overline{N_{\lambda}},$$
i.e., $g$ does not fix $\overline{N_{\lambda}}$ setwise. Therefore, for the action of $G$ on $\overline{M_{\lambda}}$, the isotropy group of $\overline{N_{\lambda}}$ is $H_{\lambda}\ll \mathcal{N} \rr$. As a consequence, for the action of $Q$ on $\overline{M_{\lambda}}$, the isotropy group of $\overline{N_{\lambda}}$ is $R_{\lambda}$.
\end{proof}

Recall that if $\mathcal{O}$ is a ring, $\mathcal{D}$ is a subring of $\mathcal{O}$, and $A$ is a $\mathcal{D}$-module, the \textit{induced module of} $A$ \textit{from} $\mathcal{D}$ \textit{to} $\mathcal{O}$, denoted as $Ind_{\mathcal{D}}^{\mathcal{O}}A$, is the tensor product $\mathcal{O}\bigotimes_{\mathcal{D}}A$. If $\mathcal{O},\mathcal{D}$ are integral group rings, we simplify notations by dropping $\mathbb{Z}$, e.g., we write $Ind^G_H$ instead of $Ind^{\mathbb{Z}G}_{\mathbb{Z}H}$. For $\lambda\in\Lambda$, Lemma \ref{permute factors}, together with the following Proposition \ref{induce}, which is a well-known characterization of induced modules (for example, see \cite[Proposition 5.3 of Chapter III]{brown1982cohomology}), implies that $\overline{M_{\lambda}}\cong Ind_{R_{\lambda}}^QRel(H_{\lambda},N_{\lambda})$ as $\mathbb{Z}Q$-modules.

\begin{proposition}\label{induce}
Let $G$ be a group and $A$ be a $\mathbb{Z}G$-module. Suppose that the underlying abelian group of $A$ is a direct sum $\bigoplus_{i\in I}A_i$ and that the $G$-action transitively permutes the summands. If $H\leqslant G$ is the isotropy group of $A_j$ for some $j\in I$. Then $A_j$ is a $\mathbb{Z}H$-module and $A\cong Ind^G_H A_j$ as $\mathbb{Z}G$-modules.
\end{proposition}

\begin{proof}[Proof of Proposition \ref{genmodule}]
For every $\lambda\in\Lambda$, $\overline{M_{\lambda}}\cong Ind_{R_{\lambda}}^QRel(H_{\lambda},N_{\lambda})$ as $\mathbb{Z}Q$-modules. Thus, $$Rel(G,\ll \mathcal{N} \rr)=\bigoplus_{\lambda\in\Lambda}\overline{M_{\lambda}}\cong \bigoplus_{\lambda\in\Lambda}Ind_{R_{\lambda}}^QRel(H_{\lambda},N_{\lambda})$$ as $\mathbb{Z}Q$-modules.
\end{proof}

\begin{example}
Let $\mathcal{G}$ be a graph of groups, $\pi_1(\mathcal{G})$ be the fundamental group of $\mathcal{G}$, $\{G_v\}_{v\in V\mathcal{G}}$ be the collection of vertex subgroups, and $\{G_e\}_{e\in E\mathcal{G}}$ be the collection of edge subgroups. By \cite[Example 4.12]{dahmani2017hyperbolically}, $\{G_v\}_{v\in V\mathcal{G}}\hookrightarrow_{wh}(\pi_1(\mathcal{G}),X)$ provided that the subset $X\subset \pi_1(\mathcal{G})$ consists of stable letters (i.e., generators corresponding to edges of $\mathcal{G}\setminus T\mathcal{G}$, where $T\mathcal{G}$ is a spanning tree of $\mathcal{G}$), and the corresponding relative metric on a vertex group $G_v$ corresponding to a vertex $v\in V\mathcal{G}$ is bi-Lipschitz equivalent to the word metric with respect to the union of the edge subgroups of $G_v$ corresponding to edges incident to $v$. Thus, we have the following corollary to Theorems \ref{Dehn filling}, \ref{main} and Proposition \ref{genmodule}.
\end{example}

\begin{corollary}\label{graph of groups}
Let $\mathcal{G}$ be a graph of groups, $\pi_1(\mathcal{G})$ be the fundamental group of $\mathcal{G}$, $\{G_v\}_{v\in V\mathcal{G}}$ be the collection of vertex subgroups, and $\{G_e\}_{e\in E\mathcal{G}}$ be the collection of edge subgroups. Suppose that, for every $v\in V\mathcal{G}$, $N_v$ is normal subgroup of $G_v$ with
$$N_v\cap\langle G_e, v\in e \rangle=\emptyset.$$
Then the triple $(G,\{G_v\}_{v\in V\mathcal{G}},\{N_v\}_{v\in V\mathcal{G}})$ has the Cohen-Lyndon property, and $Rel(G,\ll\mathcal{N}\rr)\cong \bigoplus Ind^Q_{Q_v}Rel(G_v,N_v)$ as $\mathbb{Z}Q$-modules, where $\mathcal{N}=\bigcup_{v\in V\mathcal{G}}N_v,Q=G/\ll\mathcal{N}\rr$, and $Q_v=G_v/N_v$ for $v\in V\mathcal{G}$.
\end{corollary}

In particular,

\begin{corollary}\label{amalproduct}
Let $G=A\ast_C B$ be an amalgamated free product. If $N\lhd A$ and $N\cap C=\{1\}$, then $(G,A,N)$ has the Cohen-Lyndon property, and $Rel(G,\ll N \rr)\cong Ind^Q_R Rel(A,N)$ as $\mathbb{Z}Q$-modules, where $Q=G/\ll N \rr$ and $R=A/N$.
\end{corollary}

\begin{corollary}\label{HNN}
Let $G=H\ast_t$ be an HNN-extension with associated subgroups $A,B\leqslant H$. If $N\lhd H$ and $N\cap(A\cup B)=\{1\}$, then $(G,H,N)$ has the Cohen-Lyndon property, and $Rel(G,\ll N \rr)\cong Ind^Q_R Rel(H,N)$ as $\mathbb{Z}Q$-modules, where $Q=G/\ll N \rr$ and $R=H/N$.
\end{corollary}

Alternatively, Corollary \ref{amalproduct} can be deduced from \cite{karrass1970subgroups} and both of Corollaries \ref{amalproduct},\ref{HNN} can be deduced from the Bass-Serre theory.

\bibliographystyle{alpha}
\bibliography{bin_refs}

\Addresses

\end{document}